\documentclass[12pt]{amsart}

\usepackage[a4paper,left=2.5cm,right=2.5cm,top=2.5cm,bottom=2.5cm]{geometry}

\usepackage{amsthm,amsmath,amsfonts,amssymb}
\usepackage{mathrsfs,latexsym,enumitem}
\usepackage{graphicx,caption,subcaption}
\usepackage{soul,xcolor}
\usepackage{tikz}
\usepackage{verbatim}
\usepackage{cancel}

\newtheorem{theorem}{Theorem}[section]
\newtheorem{lemma}[theorem]{Lemma}
\newtheorem{proposition}[theorem]{Proposition}
\newtheorem{corollary}[theorem]{Corollary}
\newtheorem{remark}[theorem]{Remark}
\numberwithin{equation}{section}

\newcommand{\e}{{\rm e}}
\renewcommand{\i}{{\rm i}}
\allowdisplaybreaks

\begin{document}
\title[Exact dimensionality and projections of GMC measures]{Exact dimensionality and projection properties of Gaussian multiplicative chaos measures}

\author{Kenneth Falconer}
\address{Mathematical Institute, University of St Andrews, North Haugh, St Andrews, Fife, KY16 9SS, Scotland}
\email{kjf@st-andrews.ac.uk}

\author{Xiong Jin}
\address{School of Mathematics, University of Manchester, Oxford Road, Manchester M13 9PL, United Kingdom}
\email{xiong.jin@manchester.ac.uk}

\setstcolor{red}

\begin{abstract}
Given a measure $\nu$ on a regular planar domain $D$, the Gaussian multiplicative chaos measure of $\nu$ studied in this paper is the random measure  ${\widetilde \nu}$ obtained as the limit of the exponential of the $\gamma$-parameter circle averages of the Gaussian free field  on $D$ weighted by $\nu$. We investigate the dimensional and geometric properties of these random measures. We first show that if $\nu$ is a finite Borel measure on $D$ with exact dimension $\alpha>0$, then the associated GMC measure ${\widetilde \nu}$ is non-degenerate and is almost surely exact dimensional with dimension $\alpha-\frac{\gamma^2}{2}$, provided $\frac{\gamma^2}{2}<\alpha$. We then show that if  $\nu_t$ is a H\"{o}lder-continuously parameterized family of measures then the total mass of ${\widetilde \nu}_t$ varies H\"{o}lder-continuously with $t$, provided that $\gamma$ is sufficiently small. As an application we show that if $\gamma<0.28$, then, almost surely, the orthogonal projections of the $\gamma$-Liouville quantum gravity measure ${\widetilde \mu}$ on a rotund convex domain $D$ in all directions are simultaneously absolutely continuous with respect to Lebesgue measure with H\"{o}lder continuous densities. Furthermore,  ${\widetilde \mu}$ has positive Fourier dimension almost surely.
\end{abstract}


\maketitle

\section{Introduction}

\subsection{Overview}

There has been enormous recent interest in geometrical and dimensional properties of classes of  deterministic and random fractal sets and measures. Aspects investigated include the exact dimensionality of measures, and dimension and continuity properties of  projections and sections of sets and measures and their intersection with families of curves, see for example \cite{FFJ15, SSou14} and the many references therein. 

A version of Marstrand's projection theorem \cite{Mar54} states that if a measure $\nu$ in the plane has Hausdorff dimension $\dim_H\nu>1$, then its orthogonal projection  $\pi_\theta \nu$ in direction $\theta$ is absolutely continuous with respect to Lebesgue measure except for a set of $\theta$ of Lebesgue measure 0. Considerable progress has been made recently on the challenging question of identifying classes of measures for which there are no exceptional directions, or at least for which the set of exceptional directions is very small or is identifiable.

Peres and Shmerkin \cite{PS09} and Hochman and Shmerkin \cite{HS12}, showed that for self-similar measures with $\dim_H\nu>1$ such that the rotations underlying the defining similarities generate a dense subset of the rotation group, the projected measures have dimension 1 in all directions, and Shmerkin and Solomyak \cite{SS14} showed that they are absolutely continuous except for a set of directions of Hausdorff dimension 0. Falconer and Jin \cite{FJ14,FJ15} obtained similar results for random self-similar measures and in particular their analysis included Mandelbrot's random cascade measures \cite{KP76,Man74,Per74}. Shmerkin and  Suomala \cite{SSou14} have studied such problems for certain other classes of random sets and measures. Many such geometric properties depend crucially on the measures in question being exact dimensional, that is with the local dimension $\lim_{r \to 0} \log \nu(B(x,r))/ \log r$ existing and equalling a constant for $\nu$-almost all $x$.



The main aim of this paper is to study the exact-dimensionality and absolute continuity of projections of a class of random planar measures, namely the Gaussian multiplicative chaos (GMC) measures. The GMC measures were introduced by Kahane \cite{Kah85} in 1985 as a mathematically rigorous construction of Mandelbrot's initial model of energy dissipation \cite{Man72}. The GMC measures might intuitively be thought of as continuously constructed analogues of random cascade measures, which have the disadvantages of having  preferred scales and not being isotropic or  translation invariant.  The construction has two stages. First a log-correlated Gaussian field, that is a random distribution $\Gamma$ with a logarithmic covariance structure, is defined on a planar domain $D$. Then the GMC measure is defined as a normalized exponential of $\Gamma$ with respect to a given measure supported in the domain. There are technical difficulties in this construction  since $\Gamma$ is a random Schwartz distribution rather than a random function, and this is generally addressed using smooth approximations to $\Gamma$.  Kahane used the partial sums of a sequence of independent Gaussian processes to approximate $\Gamma$ and showed the uniqueness of the GMC measure, i.e., that the law of the GMC measure does not depend on the choice of the approximating sequence. More recently, Duplantier and Sheffield \cite{DS11} constructed a GMC measure by using a circle average approximation of $\Gamma$ where $\Gamma$ is the Gaussian Free Field (GFF) on a regular planar domain $D$ with certain boundary conditions, and normalized  with respect to Lebesgue measure on $D$. They also pointed out that such a class of random measures, which is indexed by a parameter $\gamma \in [0,2)$, may be regarded as giving a rigorous interpretation of the Liouville measure that occurs in Liouville quantum gravity (LQG) and the name `$\gamma$-LQG measure' has become attached to the two-dimensional Lebesgue measure case. Surveys and further details of this area may be found in \cite{Be15,Be17,BSS14, DS11,RV14}.  

In this paper we work with an arbitrary base measure $\nu$ on $D$ (rather than just Lebesgue measure) and we denote by $\widetilde{\nu}$ the GMC measure of $\nu$ obtained as the weak limit of the circle averages of the GFF on $\nu$ which will depend on the parameter $\gamma\in [0,2)$, see Sections  \ref{sec1.2} and \ref{sec1.3}. In particular, if $\nu=\mu$ is planar Lebesgue measure on $D$ then $\widetilde{\mu}$ is the $\gamma$-LQG measure introduced by Duplantier and Sheffield in \cite{DS11}.
It is natural to study exact-dimensionality of GMC measures, along with their geometry, including their dimensions, sections and projections. 


In  Theorem \ref{locdim} of this paper we relate the dimensions of the measure $\nu$ to those of $\widetilde{\nu}$. As a corollary, if  $\nu$ is exact dimensional of dimension $\alpha>\frac{\gamma^2}{2}$ then $\widetilde{\nu}$ is exact dimensional of dimension $\alpha-\frac{\gamma^2}{2}$. Note that this result is very general and does not require further conditions other than exact dimensionality of $\nu$.
 Then, taking $\widetilde{\mu}$ to be the $\gamma$-LQG measure on $D$,  Theorem \ref{projs} asserts that if $\gamma<\frac{1}{33}\sqrt{858-132\sqrt{34}}\approx 0.28477489$ then almost surely the orthogonal projections of $\widetilde{\mu}$ in all directions are {\it simultaneously} absolutely continuous with respect to one-dimensional Lebesgue measure. A consequence, Corollary \ref{ft}, is that for such $\gamma$,  the $\gamma$-LQG measure $\widetilde{\mu}$ almost surely has positive Fourier dimension. 
These last results follow from a much more general Theorem \ref{thm1.1} which  shows that for suitable families of measures  $\{\nu_t:t\in \mathcal{T}\}$ on  ${\overline D}$ with a H\"{o}lder continuous parameterization by a metric space $\mathcal{T}$, almost surely $\|\widetilde{\nu}_{t}\|$ 
is  H\"{o}lder continuously in the  parameter $t$, where $\|\cdot\|$ denotes the total mass of a measure.  Theorem \ref{thm1.1} has many other applications, including Theorem \ref{ql}, that if we define GMC measures simultaneously on certain parameterized families of  planar curves in $D$, their mass, which may be thought of as the `quantum length of the curves', varies  H\"{o}lder continuously.
In another direction, Theorem \ref{ssm} shows that the total mass of GMC measures of self-similar measures is H\"{o}lder continuous in the underlying similarities.

The proof of the H\"older continuity of $\{\|\widetilde{\nu}_t\|:t\in\mathcal{T}\}$ in Theorem \ref{thm1.1} is inspired by the paper \cite{SSou14} of Shmerkin and Suomala on H\"older properties of `compound Poisson cascade' types of random  measures first introduced by Barral and Mandelbrot \cite{BM02}. The difference here is that the circle averages of the GFF does not have the spatial independence or the uniform bounded density properties needed in \cite{SSou14}. Hence we adopt a different approach, using a Kolmogorov continuity type argument to deduce  the H\"older continuity of  $\widetilde{\nu}_t$ from  the  convergence exponents of the approximating circle averages. It may be possible to relax some of the conditions required in \cite{SSou14} using our approach.

\subsection{Gaussian Free Fields}\label{sec1.2}

Let $D$ be a bounded regular planar domain, namely a simply-connected bounded open subset of $\mathbb{R}^2$ with a regular boundary, that is, for every point $x\in \partial D$ there exists a continuous path $u(t)$, $0\le t \le 1$, such that $u(0)=x$ and $u(t)\in D^c$ for $0<t \le 1$.  The Green function $G_D$ on $D\times D$ is given by
\[
G_D(x,y)= \log \frac{1}{|x-y|}-E^{x}\left(\log \frac{1}{|W_T -y|}\right),
\]
where the expectation $E^x$ is taken with respect to the probability measure $P^x$ under which $W$ is a planar Brownian motion started from $x$, and $T$ is the first exit time of $W$ in $D$, i.e., $T=\inf\{t\ge 0: W_t \not \in D\}$. The Green function is conformally invariant in the sense that if $f: D\mapsto D'$ is a conformal mapping, then
\[
G_D(x,y)=G_{f(D)}(f(x),f(y)).
\]
Let $\mathcal{M}^+$ be the set of finite measures $\rho$ supported in $D$ such that
\[
\int_D\int_D G_D(x,y) \, \rho(\mathrm{d}x) \rho(\mathrm{d}y)<\infty.
\]
Let $\mathcal{M}$ be the vector space of signed measures $\rho^+-\rho^-$, where $\rho^+,\rho^-\in\mathcal{M}^+$. Let $\{\Gamma(\rho)\}_{\rho\in\mathcal{M}}$ be a centered Gaussian process on $\mathcal{M}$ with covariance function
\[
\mathbb{E}(\Gamma(\rho)\Gamma(\rho'))=\int_D\int_D G_D(x,y) \,\rho(\mathrm{d}x) \rho'(\mathrm{d}y).
\]
Then $\Gamma$ is called a {\it Gaussian free field} (GFF) on $D$ with zero (Dirichlet) boundary conditions.

Let $O$ be a regular subdomain of $D$. Then  $\Gamma$ may be decomposed into a sum:
\begin{equation}\label{indep}
\Gamma=\Gamma^O+\Gamma_O,
\end{equation}
where $\Gamma^O$ and $\Gamma_O$ are two independent Gaussian processes on $\mathcal{M}$ with covariance functions $G_O$ and $G_D-G_O$ respectively. Moreover, there is a version of the process such that $\Gamma^O$ vanishes on all measures supported in $D\setminus O$, and $\Gamma_O$ restricted to $O$ is harmonic, that is there exists a harmonic function $h_O$ on $O$ such that for every measure $\rho$ supported in $O$,
\[
\Gamma_O(\rho)=\int_{O} h_O(x) \, \rho(\mathrm{d}x).
\]
In fact $h_O(x)=\Gamma(\tau_{O,x})$ for $x\in O$, where $\tau_{O,x}$ is the exit distribution of $O$ for a Brownian motion started from $x$. Furthermore, if we denote by $\mathcal{F}_{D\setminus O}$ the $\sigma$-algebra generated by all $\Gamma(\rho)$ for which $\rho\in \mathcal{M}$ is supported by $D\setminus O$, then $\Gamma^O$ is independent of $\mathcal{F}_{D\setminus O}$.

For more details on Gaussian free fields, see, for example, \cite{Be15,RV14,S07,W14}.

\subsection{Circle averages of GFF and GMC measures}\label{sec1.3}

For $x\in D$ and $\epsilon>0$ let $\rho_{x,\epsilon}$ be Lebesgue measure on $\{y\in D: |x-y|=\epsilon\}$, the circle centered at $x$ with radius $\epsilon$ in $D$, normalised to have mass $1$. Fix $\gamma\ge 0$. Let $\nu$ be a finite Borel measure supported in $D$. For integers $n\ge 1$ let
\begin{equation}\label{muep}
\widetilde{\nu}_{n}(dx)=2^{-n\gamma^2/2} \e^{\gamma \Gamma(\rho_{x,2^{-n}})} \, \nu(\mathrm{d}x), \quad x\in D.
\end{equation}
Then the almost sure weak limit 
\begin{equation}\label{cirlim}
\widetilde{\nu}= \mbox{\rm w-}\!\lim_{n\to \infty} \widetilde{\nu}_n,
\end{equation}
whenever it exists, is called a {\it Gaussian multiplicative chaos (GMC) measure} of $\nu$.

We write $\mu$ for the important case of  planar Lebesgue measure restricted to $D$. When $\gamma\in[0,2)$ the GMC measure $\widetilde{\mu}$ exists and is non-degenerate, and  is called the $\gamma${\em -LQG measure} on $D$. For more details on $\gamma$-LQG measures, see  for example \cite{Be15,DS11}.

Since $\Gamma(\rho_{x,\epsilon})$ is centered Gaussian,
\begin{equation}\label{varexp}
\mathbb{E}\left(\e^{\gamma \Gamma(\rho_{x,\epsilon})}\right)=\e^{\frac{\gamma^2}{2} \mathrm{Var}(\Gamma(\rho_{x,\epsilon}))}.
\end{equation}
Using the conformal invariance of GFF it can be shown that, provided that $B(x,\epsilon)  \subset D$, where  $B(x,\epsilon)$ is the open ball of centre $x$ and radius $\epsilon$,
\begin{equation}\label{exp}
\mathrm{Var}(\Gamma(\rho_{x,\epsilon}))=-\log \epsilon+\log R(x, D),
\end{equation}
where $R(x,D)$ is the conformal radius of $x$ in $D$, given by $R(x,D)=|f'(0)|$ where $f: \mathbb{D} \mapsto D$ is a conformal mapping from the unit disc $\mathbb{D}$ onto $D$ with $f(0)=x$. Then for all $\gamma \geq 0$, if $B(x,\epsilon)  \subset D$,
\begin{equation}\label{eeg}
\mathbb{E}\left(\e^{\gamma \Gamma(\rho_{x,\epsilon})}\right)=\epsilon^{-\gamma^2/2}R(x,D)^{\gamma^2/2},
\end{equation}
and so
\[
\mathbb{E}(\widetilde{\nu}(dx))=R(x,D)^{\gamma^2/2} \nu(dx), \quad x\in D.
\]
It is well-known that $R(x,D)$ is comparable to ${\rm dist}(x, \partial D)$, the distance from $x$ to the boundary of $D$, indeed, using the Schwarz lemma and the Koebe 1/4 theorem, 
\begin{equation}\label{confrad}
{\rm dist}(x, \partial D) \leq R(x,D) \leq 4\, {\rm dist}(x, \partial D).
\end{equation}

\section{Main results}

Throughout the paper we shall make the following assumption (A0) on the regularity of the boundary of $D$: For $n\ge 1$ and $m_1,m_2 \in \mathbb{Z}$ let
\[
S_{m_1,m_2}^n=[m_1 2^{-n}, (m_1+1)2^{-n}) \times [m_2 2^{-n}, (m_2+1)2^{-n})
\]
denote a square in $\mathbb{R}^2$ of side-lengths $2^{-n}$ with respect to some pair of coordinate axes. Let $D$ be a fixed bounded regular planar domain. For $n\ge 1$ let $\mathcal{S}_n$ be the family of sets
\[
\big\{D\cap S_{m_1,m_2}^n: m_1,m_2 \in \mathbb{Z},D\cap S^n_{m_1,m_2}\neq \emptyset\big\}.
\]
For $S=D\cap S_{m_1,m_2}^n \in \mathcal{S}_n$ denote by
\begin{equation}\label{stilde}
\widetilde{S}=D\cap \big( [(m_1-1) 2^{-n}, (m_1+2)2^{-n}) \times [(m_2-1) 2^{-n}, (m_2+2)2^{-n}) \big)
\end{equation}
the $3$-fold enlargement of $S$ in $D$. Our assumption states as follows.
\begin{itemize}
\item[(A0)] There exists an integer $N_0$ such that for $n\ge N_0$ the enlargement $\widetilde{S}$ is simply connected for all $S\in \mathcal{S}_n$, and for $x\in D$ there exists $y\in D$ with $|x-y|\le 2^{-n+1}$ such that $B(y,2^{-n})\subset D$.
\end{itemize}
In particular (A0) is satisfied when $D$ is a convex set with a smooth boundary. As we may rescale $D$ to be large enough, without loss of generality we may take $N_0=1$.

\subsection{Exact dimension results}
Let $\nu$ be a finite Borel measure supported in $D$, let $\gamma>0$ and define the GMC measure
$\widetilde{\nu}$ by \eqref{muep} and \eqref{cirlim}. The following theorem relates the local behaviour of $\nu$ to that of $\widetilde{\nu}$. (Note that if $\alpha = \beta$ in \eqref{mesgrow} and \eqref{mesgrow2} then $\nu$ is termed $\alpha$-Alhfors regular.)

\begin{theorem}\label{locdim}
Assume that $D$ satisfies (A0). Suppose that there exist constants $C_0>0$, $r_0>0$ and $ \alpha>\frac{\gamma^2}{2}$ such that
\begin{equation}
 \nu(B(x,r))\le C_0 r^\alpha \label{mesgrow}
\end{equation}
for all $x\in\mathrm{supp}(\nu)$ and $r\in(0,r_0)$. Then, almost surely $\widetilde{\nu}_{n}$ converges weakly to a non-trivial limit measure $\widetilde{\nu}$, and for $\widetilde{\nu}$-a.e. $x$,
\[
\liminf_{r\to 0} \frac{\log \widetilde{\nu}(B(x,r))}{\log r}\ge \alpha-\frac{\gamma^2}{2}.
\]
In the opposite direction, if there exists a constant $\beta\ge \alpha$ such that
\begin{equation}\label{mesgrow2}
\nu(B(x,r))\ge C_0^{-1}r^\beta
\end{equation}
for all $x\in\mathrm{supp}(\nu)$ and $r\in(0,r_0)$, then  for $\widetilde{\nu}$-a.e. $x$,
\[
\limsup_{r\to 0} \frac{\log \widetilde{\nu}(B(x,r))}{\log r} \le \beta-\frac{\gamma^2}{2}.
\]
\end{theorem}

\begin{remark}
The almost sure convergence of $\widetilde{\mu}_{n}$ to $\widetilde{\mu}$ when $\mu$ is Lebesgue measure on $D$ was established in \cite{DS11}. This, and the convergence part of Theorem \ref{locdim}, are not directly covered by Kahane's multiplicative chaos theory approach as the circle averages of GFFs, although they can be written as a sum of independent random variables at individual points, cannot be decomposed into a sum of independent random fields on $D$. 
\end{remark}

Recall that a Borel measure $\nu$  is {\it exact-dimensional of dimension} $\alpha$ if 
$$ \lim_{r\to 0} \frac{\log{\nu}(B(x,r))}{\log r}=\alpha,$$
with the limit existing, $\nu$-almost everywhere. 
The Hausdorff dimension of a measure $\nu$ is given by
$$\dim_H \nu = \inf\big\{\dim_H E: E \mbox{ is a Borel set with } \nu(E)>0\big\};$$
in particular, $\dim_H \nu = \alpha$ if  $\nu$  is exact-dimensional of dimension $\alpha$, see \cite{Fa97}.

A variant of  Theorem \ref{locdim} gives the natural conclusion for exact-dimensionality.


\begin{corollary}\label{cored}
Assume that $D$ satisfies (A0). If $\nu$ is exact-dimensional with dimension $\alpha>\frac{\gamma^2}{2}$, then, the GMC measure $\widetilde{\nu}$ of $\nu$ is well-defined and non-trivial, and almost surely, $\widetilde{\nu}$ is exact-dimensional with dimension $\alpha-\frac{\gamma^2}{2}$. 
\end{corollary}

This corollary applies to the large class of measures that are exact dimensional, including self-similar measures and, more generally, Gibbs measures on  self-conformal sets, see \cite{FJ14, FeHu09}, as well as planar self-affine measures \cite{BK16}.

\begin{remark}
The assumption (A0) in Corollary \ref{cored} can be relaxed. One can work with domains that can be decomposed into pieces of subdomains where each subdomain can be approximated from within by convex sets with smooth boundaries.
\end{remark}

\subsection{Absolute continuity of projections}\label{secabscty}
We write 
$\pi_\theta$ for the orthogonal projection onto the line through the origin in direction  perpendicular to the unit vector $\theta$, and  $\pi_\theta\rho = \rho\circ \pi_\theta^{-1}$ for the projection of a measure $\rho$ on $\mathbb{R}^2$ in the obvious way.

It follows from the work of Hu and Taylor \cite{HT94} and Hunt and Kaloshin \cite{HK97} that if a Borel measure $\rho$ on $\mathbb{R}^2$ is exact dimensional of dimension $\alpha$, then for almost all $\theta \in [0,\pi)$, the projected measure $\pi_\theta \rho$ is exact dimensional of dimension $\min\{1,\alpha\}$. Moreover, if $\alpha>1$ then  $\pi_\theta \rho$ is absolutely continuous with respect to Lebesgue measure for almost all $\theta$. In particular this applies to the projections of the  GMC measures obtained in Corollary \ref{cored}.

The $\gamma$-LQG measure $\widetilde{\mu}$, obtained from circle averages of GFF acting on planar Lebesgue measure $\mu$ on $D$, is almost surely exact dimensional of dimension $\dim_H \widetilde{\mu} =2-\frac{\gamma^2}{2}$, in which case for almost all $\theta$ the 
orthogonal projections $\pi_\theta \widetilde{\mu}$  are exact dimensional of dimension $\min\{1,\dim_H \widetilde{\mu}\}$ and are absolutely continuous if $0<\gamma <\sqrt{2}\approx 1.4142$. Here we show that, for suitable domains $D$, if $0<\gamma <\frac{1}{33}\sqrt{858-132\sqrt{34}}\approx 0.28477489$  then for {\it all} $\theta$ {\it simultaneously}, not only are the projected measures $\pi_\theta \widetilde{\mu}$ absolutely continuous with respect to Lebesgue measure but also the Radon-Nikodym derivatives are $\beta$-H\"{o}lder for some $\beta>0$. Note that, according to Rhodes and Vargas \cite{RV14}, the support of the multifractal spectrum of $\widetilde{\mu}$ is the interval $\big[\big(\sqrt{2} - \frac{\gamma}{\sqrt{2}}\big)^2, \big(\sqrt{2} + \frac{\gamma}{\sqrt{2}}\big)^2\big]$, meaning that the smallest possible local dimension of $\widetilde{\mu}$ is $\big(\sqrt{2} - \frac{\gamma}{\sqrt{2}}\big)^2$, so in particular the projected measures can only be absolutely continuous with continuous  Radon-Nikodym derivatives if $\gamma \leq 2 -\sqrt{2} \approx  0.5858$ as  otherwise $\widetilde{\mu}$ has points of local dimension less than 1. Whilst we would expect absolute continuity with continuous  Radon-Nikodym derivatives for all projections simultaneously for all $\gamma <2 -\sqrt{2} $, this would require significantly new methods to establish. On the other hand, if we just require the projections to be absolutely continuous, it may be enough for the dimension of the LQG measure to be larger than 1, corresponding to $\gamma <\sqrt{2}$. Again, it would be nice to obtain good estimates for the H\"{o}lder exponent $\beta$ but, whilst our method might be followed through to obtain positive lower bounds for $\beta$, such estimates are likely to be small. These questions are considered further at the end of Section \ref{sec4}

 We call a bounded open convex domain $D \subset \mathbb{R}^2$ {\it rotund} if its boundary $\partial D$ is twice continuously differentiable with radius of curvature bounded away from $0$ and $\infty$.
 
\begin{theorem}\label{projs}
Let $0<\gamma <\frac{1}{33}\sqrt{858-132\sqrt{34}}$ and let $\widetilde{\mu}$ be $\gamma$-LQG on a rotund convex domain $D$. Then, almost surely,  for all $\theta \in [0,\pi)$ the projected measure $\pi_\theta \widetilde{\mu}$ is absolutely continuous with respect to Lebesgue measure with a $\beta$-H\"{o}lder continuous Radon-Nikodym derivative for some $\beta>0$.
\end{theorem}

Theorem \ref{projs} follows from a much more general result on the H\"{o}lder continuity of parameterized families of measures given as Theorem \ref{thm1.1} below. We remark that Theorem \ref{thm1.1} also implies that for a given fixed $\theta$ the projected measure $\pi_\theta \widetilde{\mu}$ has a H\"{o}lder continuous  density for the larger range $0<\gamma <\frac{1}{17}\sqrt{238-136\sqrt{2}}\approx 0.3975137$, see the comment after the proof of Theorem \ref{projs} in Section 5.

Theorem \ref{projs} leads to a bound on the rate of decay of the Fourier transform $\widehat{\widetilde{\mu}}$ of $\widetilde{\mu}$, or, equivalently, on the Fourier dimension of the measure defined as the supremum value of $s$ such that  $|\widehat{\widetilde{\mu}}(\xi)| \leq C |\xi|^{-s/2} \quad ( \xi \in \mathbb{R}^2)$ for some constant $C$; see \cite{EPS14, Mat15} for recent discussions on Fourier dimensions. One might conjecture that, as is fairly typical  for random measures, the Fourier dimension of the LQG measure equals its Hausdorff dimension for all $0<\gamma< 2-\sqrt{2}$. However, Fourier dimensions can be very difficult to estimate and even demonstrating that they are positive is often non-trivial.

\begin{corollary}\label{ft}
Let $0<\gamma <\frac{1}{33}\sqrt{858-132\sqrt{34}}$, let $\widetilde{\mu}$ be  $\gamma$-LQG on a rotund convex domain $D$ and let $\beta>0$ be given in Theorem \ref{projs}. Then, almost surely, there is a random constant $C$ such that
\begin{equation}\label{ftest}
|\widehat{\widetilde{\mu}}(\xi)| \leq C |\xi|^{-\beta}, \qquad \xi \in \mathbb{R}^2,
\end{equation}
so in particular $\widetilde{\mu}$ has Fourier dimension at least $2\beta>0$.
\end{corollary}

\subsection{Parameterized familes of measures}\label{assumps}
We now state our main result on the H\"{o}lder continuity of  the total masses of the GMC measures of certain parameterized families of measures, typically measures on  parameterized families of planar curves. First we set up the notation required and state some natural assumptions that we make.

Let $(\mathcal{T},d)$ be a compact metric space which will parameterize lines or other subsets of $D$. Let $\nu$ be a positive finite measure on a measurable space $(E,\mathcal{E})$.  For each $t\in \mathcal{T}$ we assign  a measurable set $ I_t\in\mathcal{E}$, a Borel set $L_t\subset \overline{D}$ and a measurable function $f_t$,
\[
f_t:I_t\to L_t,
\]
and define the push-forward measure on $D$ by 
\[
\nu_t:=\nu\circ f_t^{-1},
\]
with the convention that $\nu_t$ is the null measure if $\nu(I_t)=0$. To help fix ideas,  $I_t$ may typically be a real interval with $f_t$ a continuous injection, so that $L_t$ is a curve in $\overline{D}$ that supports the measure $\nu_t$.

We make the following three assumptions: (A1) is a bound on the local dimension of the measures $\nu_t$, (A2) is a H\"{o}lder condition on the $f_t$ and thus on the $\nu_t$, and (A3) means that the parameter space $(\mathcal{T},d)$ may be represented as a bi-Lipschitz image of a convex set in a finite dimensional Euclidean space. (In fact (A3) can be weakened considerably at the expense of simplicity, see Remark \ref{rema3}.)
\begin{itemize}
\item[(A1)] There exist constants $C_1,\alpha_1>0$ such that for all $x \in \mathbb{R}^2$ and $r>0$,
\[
\sup_{t\in\mathcal{T}}\nu_t(B(x,r))\le C_1r^{\alpha_1};
\]

\item[(A2)] There exist constants $C_2,r_2,\alpha_2,\alpha_2'\ge 0$ such that for all $s,t\in \mathcal{T}$ with $d(s,t)\le r_2$ and $I_s\cap I_{t} \neq \emptyset$,
\[
\sup_{u\in I_s\cap I_{t}}|f_s(u)-f_t(u)|\le C_2d(s,t)^{\alpha_2}
\]
and 
\[
\nu(I_s\Delta I_t) \le C_2 d(s,t)^{\alpha_2'}.
\]

\item[(A3)] There exist a convex set $G \subset [0,1]^k$ with non-empty interior for some $k\ge 1$, a one-to-one map $g:\mathcal{T}\mapsto G$ and a constant $0<C_3<\infty$ such that for all $s,t\in \mathcal{T}$,
\[
C_3^{-1} d(s,t)\le |g(s)-g(t)| \le C_3 d(s,t).
\]
\end{itemize}


For $t\in \mathcal{T}$ and $n\ge 1$ we define circle averages of $\Gamma$ on $\nu_t$  by
\begin{equation}\label{defnutn}
\widetilde{\nu}_{t,n}(dx)=2^{-n\gamma^2/2} \e^{\gamma \Gamma(\rho_{x,2^{-n}})} \, \nu_t(dx), \quad  x\in D,
\end{equation}
and let
\begin{equation}\label{defytn}
Y_{t,n}:=\|\widetilde{\nu}_{t,n}\|
\end{equation}
be the total mass of $\widetilde{\nu}_{t,n}$. Let $\widetilde{\nu}_{t}= \mbox{\rm w-}\!\lim_{n\to\infty} \widetilde{\nu}_{t,n}$ be the GMC of $\nu_t$ and $Y_{t}=\|\widetilde{\nu}_{t}\|$ be its total mass if it exists. (Taking circle averages with dyadic radii $\epsilon = 2^{-n}$ does not affect the weak limit.)


Here is our main result on parameterized families of measures. For $\gamma, \lambda>0$ write
\[
n(\lambda,k,\gamma)=\Big(\frac{4k^2-\lambda k}{\lambda^2}\Big)\gamma^2+\frac{2k}{\lambda^2}\gamma\sqrt{4k^2\gamma^2+2k(1-\gamma^2)\lambda}+\frac{k}{\lambda}
\]
and for $\alpha,\gamma>0$ write
\[
m(\alpha,\gamma)=
\frac{1}{2}\Big(\frac{\alpha}{\gamma}-\frac{\gamma}{2}\Big)^2.
\]

\begin{theorem}\label{thm1.1}
Let $D$ satisfy {\rm (A0)} and let $\mathcal{T},\nu$ and the $f_t$ satisfy {\rm (A1), (A2)} and {\rm(A3)}. Write $\lambda=\alpha_2\wedge (2\alpha_2')$. If $\alpha_1>\frac{\gamma^2}{2}$, $k\ge \frac{\lambda}{2}$ and
\begin{equation}\label{mcond}
n(\lambda,k,\gamma)<m(\alpha_1,\gamma),
\end{equation}
then almost surely the sequence of mappings $\{t\mapsto Y_{t,n}\}_{n=1}^\infty$ converges uniformly on $(\mathcal{T},d)$ to a limit $t\mapsto Y_t$. Moreover, $Y_t$ is $\beta$-H\"older continuous in $t$ for some $\beta>0$.
\end{theorem}

\begin{remark}\label{rmkgam}
The condition $\alpha_1>\frac{\gamma^2}{2}$ ensures that each GMC measure $\widetilde{\nu}_{t}$ is non-degenerate. It is easy to see that when $\gamma\to 0$,
\[
n(\lambda,k,\gamma)\to \frac{k}{\lambda} \text{ and } m(\alpha_1,\gamma)\to \infty.
\]
Therefore, given $\alpha_1$, $\alpha_2$, $\alpha_2'$, $k$ \eqref{mcond} will always hold if $\gamma>0$ is small enough.
For specific $\alpha_1$, $\alpha_2$, $\alpha_2'$, $k$ one can derive a range  $0<\gamma< \gamma_{\max}$ over which this condition is satisfied. Whilst this often gives a reasonable range of  $\gamma$, it is unlikely to be best possible given the lack of sharpness in Lemma \ref{lemma2}, see the end of Section 4 for a further discussion on this.
\end{remark}

As we shall see, Theorem \ref{projs} follows from Theorem \ref{thm1.1} on taking $\nu_t$ to be 1-dimensional Lebesgue measure restricted to chords of $D$ which are parameterized by their direction and displacement from some origin.

The many applications of Theorem \ref{thm1.1} include quantum length on families of planar curves and  quantum masses of self-similar measures.

\subsubsection{Quantum length of planar curves}\label{secquantlen}

Let $D$ satisfy (A0). Let $\mathcal{T}=[0,T]$ and let $d$ be Euclidean distance on $\mathcal{T}$. Let $f:\mathcal{T}\to D$ be a measurable function. Note here that we do not need to assume $f$ to be continuous. For $t\in[0,T]$ let $I_t=[0,T]$ and $f_t=f|_{I_t}$.  Let $\nu$ be the one-dimensional Lebesgue measure on $[0,T]$. If we assume that the occupation measure $\nu\circ f^{-1}$ satisfies
\[
\nu\circ f^{-1}(B(x,r)) \le C r^{\alpha}
\]
for some $C>0$ and $0<\alpha\le 2$. Then we may take $\alpha_1=\alpha$, $\alpha_2$ arbitrarily large, $\alpha_2'=1$ and $k=1$ in assumptions (A1), (A2) and (A3).  In such a case we have $\lambda=2$ and \eqref{mcond} becomes
\[
\frac{1}{2}\gamma^2+\gamma+\frac{1}{2}<\frac{1}{2}\Big(\frac{\alpha}{\gamma}-\frac{\gamma}{2}\Big)^2.
\]
Since $0<\gamma<\sqrt{2\alpha}$, the above inequality is equivalent to
\[
\gamma+1<\frac{\alpha}{\gamma}-\frac{\gamma}{2},
\]
which means
\[
\gamma<\frac{\sqrt{6\alpha+1}-1}{3}.
\]
In this context Theorem \ref{thm1.1} immediately translates into the following result.

\begin{theorem}\label{ql}
Let $D$ satisfy (A0). Let $f:[0,T]\to D$ be a measurable function such that the occupation measure $\nu\circ f^{-1}$ satisfies
\[
\nu\circ f^{-1}(B(x,r)) \le C r^{\alpha} \text{ for all } x\in D \text{ and } r>0
\]
for some $C>0$ and $0<\alpha\le 2$, where $\nu$ denotes the Lebesgue measure on $[0,T]$. For $t\in[0,T]$ denote by $\nu_t=\nu|_{[0,t]}\circ f^{-1}$ and let $\big\{\widetilde{\nu}_{t}:t\in[0,T]\big\}$ be the corresponding GMC measures of  $\big\{\nu_{t}:t\in [0,T]\big\}$ with parameter
\[
\gamma<\frac{\sqrt{6\alpha+1}-1}{3}.
\]
Then, almost surely,  the function
\[
L:[0,T]\ni t\mapsto \|\widetilde{\nu}_t\|
\]
is H\"{o}lder continuous.
\end{theorem}

\noindent{\bf Example 1:} Let $f:[0,T]\to D$ be a smooth curve in $D$. Then we may take $\alpha=1$. In this case the `$\gamma$-quantum length' of $f$ is H\"older continuous when
\[
\gamma<\frac{\sqrt{7}-1}{3} \approx 0.5485837.
\]
This $\gamma$-quantum length is slightly different to that in \cite{S16} introduced by Sheffield. In \cite{S16} the boundary LQG $\widetilde{\nu}$ is defined as the exponential of the semi-circle average of the GFF with free boundary condition in the upper-half plane with respect to one-dimensional Lebesgue measure on the boundary $\mathbb{R}$, and the quantum boundary lengths considered there are $\widetilde{\nu}([0,t])$ and $\widetilde{\nu}([-t,0])$ for $t\ge 0$. Sheffield shows that the `conformal welding' or `conformal zipping' of $\widetilde{\nu}([0,t])$ and $\widetilde{\nu}([-t,0])$ is actually a SLE curve, resolving a conjecture of Peter Jones. As the boundary LQG $\widetilde{\nu}$ is very similar to a one-dimensional GMC with respect to Lebesgue measure, the H\"older continuity of $t\to \widetilde{\nu}([0,t])$ for all parameters $0<\gamma<\sqrt{2}$ may be deduced from its $p$-moment control  ($p>1$)  and Kolmogorov continuity type arguments, as in \cite{BJM10} for multiplicative cascades.

\medskip

\noindent{\bf Example 2:} Let $f:[0,\tau]\to D$ be a segment of planar Brownian motion in $D$. It is well-known (see \cite{LeGall92} for example) that we may take $\alpha$ arbitrarily close to $2$ for the occupation measure of planar Brownian motion. This implies that the `$\gamma$-quantum length' of planar Brownian motion is H\"older continuous when
\[
\gamma<\frac{\sqrt{13}-1}{3} \approx 0.8685171.
\]
This $\gamma$-quantum length  is used in \cite{GRV16} to define `Liouville Brownian motion'. In fact in \cite{GRV16} the authors show that this $\gamma$-quantum length is $\alpha$-H\"older continuous for all $\alpha<(1-\frac{\gamma}{2})^2$ for all $0<\gamma<2$. The proof of this nearly sharp result relies heavily on the fact that the occupation measure of planar Brownian motion is stationary under translation and it also satisfies a scaling invariance property, which we can not expect to have for general measurable functions $f$.

\begin{remark}
In both Examples 1 and 2 we have not obtained the H\"older continuity for all possible parameters $0<\gamma<\sqrt{2\alpha}$ for the occupation measure of a given planar function $f$ with dimension at least $0<\alpha\le 2$. The main reason is that a grid partition of the time parameter space $[0,T]$ does not necessarily yields a partition of its image through $f$, which causes problems in computing the moments of the associated GMC measures. In particular the moment estimates in Lemma \ref{lemma2} are not as sharp as for the classical moment estimates in Gaussian multiplicative chaos theory such as in Lemma \ref{lemma1}. Currently we do not know how to improve Lemma \ref{lemma2} to get a sharper estimate.
\end{remark}

\subsubsection{ GMC  measures on families of self-similar sets}\label{secsss}

Another application of Theorem \ref{thm1.1} gives the H\"{o}lder continuity of the total masses of the GMC measures of parameterized self-similar measures. Let $m\ge 2$ be an integer. Let $\mathcal{U} = (0,1)^m\times SO(\mathbb{R},2)^m\times (\mathbb{R}^2)^m$ be endowed with the product metric $d$. For each $t=(\vec{r},\vec{O},\vec{x}) \in \mathcal{S}$ the set of $m$ mappings
\[
\mathcal{I}_t=\big\{g_i^t(\cdot)=r_{i}O_{i}(\cdot)+x_{i}:1\le i \le m\big\}
\]
forms an iterated function system (IFS) of contracting similarity mappings. Such an IFS defines a unique non-empty compact set $F_t \subset \mathbb{R}^2$ that satisfies 
$F_t = \bigcup_{i=1}^m g_i^t(F_t)$, known as a {\it self-similar set}, see, for example, \cite{Fal14} for details of IFSs and self-similar sets and measures. 
Let $E=\{1,\ldots,m\}^\mathbb{N}$ be the symbolic space endowed with the standard product topology and Borel $\sigma$-algebra $\mathcal{E}$. In the usual way, the points of $F_t$ are coded by the canonical projection $f_t: E \to F_t$ given by 
\[
f_t(\underline{i})=f_t(i_1i_2\cdots)=\lim_{n\to\infty} g^t_{i_1}\circ\cdots\circ g^t_{i_n}(x_0),
\]
which is independent of the choice of $x_0\in \mathbb{R}^2$.

Let $\nu$ be a Bernoulli measure on $E$ with respect to a probability vector $p=(p_1,\ldots,p_m)$.  For $t\in \mathcal{S}$ let $\nu_t=\nu\circ f_t^{-1}$; then $\nu_t$ is a self-similar probability measure on $ \mathbb{R}^2$ in the sense that 
$\nu_t=\sum_{i=1}^m p_i \, \nu_t \circ (g_i^t)^{-1}$. 

Let $D\subset   \mathbb{R}^2$ be a rotund convex domain. Let $\mathcal{T}$ be a convex compact subset of $\mathcal{U}$ (with respect to some smooth Euclidean parameterization)  such that  for all $t\in \mathcal{T}$, $F_t\equiv f_t(E)\subset D$ and the open set condition (OSC) is satisfied, that is there exists a non-empty open set $U_t$ such that $U_t \supset \bigcup_{i=1}^m g_i^t(U_t)$ with this union disjoint.


  

\begin{theorem}\label{ssm}
Assume that $\gamma$ satisfies \eqref{mcond} with $\alpha_2 =1$, $\alpha'_2$ arbitrary, $k =4m$ and
\[
\alpha_1 = \min_{t\in\mathcal{T},1\leq i \leq m} \log p_i / \log r_{i};
\]
by Remark \ref{rmkgam} this will be the case if $\gamma>0$ is sufficiently small. Let $\{\widetilde{\nu}_{t}:t\in\mathcal{T}\}$ be the GMC measures of the family of self-similar measures $\{\nu_{t}:t\in \mathcal{T}\}$. Then, almost surely, the function
\[
L:\mathcal{T}\ni t\mapsto \|\widetilde{\nu}_t\|
\]
is $\beta$-H\"{o}lder continuous for some $\beta>0$.
\end{theorem}

\begin{remark}
Theorem \ref{ssm} can be naturally extended to Gibbs measures on a H\"older continuously parameterized family of self-conformal sets, such as families of Julia sets in complex dynamical systems.
\end{remark}

%

\section{Exact dimensionality proofs}

In this section we prove Theorem \ref{locdim},  first obtaining lower estimates for local dimensions in Proposition \ref{locdimleft} and then upper estimates in Proposition \ref{locdimright}. First we present the following lemma that removes the restriction of $B(x,2^{-n})\subset D$ in \eqref{eeg}. 

\begin{lemma}\label{explemma}
There exists a constant $C_{D}$ depending only on $D$ such that for $\gamma\ge 0$, for every $x\in D$ and $n\ge 1$, if there exists $y\in D$ with $|x-y|\le 2^{-n+1}$ such that $B(y,2^{-n})\subset D$, then
\[
\mathbb{E}\left(\e^{\gamma \Gamma(\rho_{x,2^{-n}})}\right) \le (C_{D})^{\gamma^2}|D|^{\frac{\gamma^2}{2}} 2^{n\frac{\gamma^2}{2}},
\]
where $|D|$ stands for the diameter of $D$.
\end{lemma}

\begin{proof}
From the proof of \cite[Proposition 2.1]{HMP10} there exists a constant $C$ depending only on $D$ such that for all $x,y\in D$ and $\epsilon,\epsilon'>0$,
\[
\mathbb{E}(|\Gamma(\rho_{x,\epsilon})-\Gamma(\rho_{y,\epsilon'})|^2) \le C\frac{|(x,\epsilon)-(y,\epsilon')|}{\epsilon\vee \epsilon'}.
\]
This implies that for all $x,y\in D$ and $\epsilon,\epsilon'>0$,
\begin{equation}\label{eeq2}
\mathrm{Var}(\Gamma(\rho_{x,\epsilon}))\le \mathrm{Var}(\Gamma(\rho_{y,\epsilon'}))+C\frac{|(x,\epsilon)-(y,\epsilon')|}{\epsilon\wedge \epsilon'}.
\end{equation}
For $x\in D$ and $n\ge 1$, let $y\in D$ be such that $|x-y|\le 2^{-n+1}$ and $B(y,2^{-n})\subset D$. Then by \eqref{eeq2},
\[
\mathrm{Var}(\Gamma(\rho_{x,2^{-n}})) \le  \mathrm{Var}(\Gamma(\rho_{y,2^{-n}}))+2C.
\]
By \eqref{varexp}, \eqref{exp} and \eqref{confrad}, this implies that
\begin{align*}
\mathbb{E}\left(\e^{\gamma \Gamma(\rho_{x,2^{-n}})}\right)=&\e^{\frac{\gamma^2}{2} \mathrm{Var}(\Gamma(\rho_{x,2^{-n}}))}\\
\le& \e^{\frac{\gamma^2}{2} (\mathrm{Var}(\Gamma(\rho_{y,2^{-n}}))+2C)}\\
=& \e^{C\gamma^2} R(y,D)^{\frac{\gamma^2}{2}} 2^{n\frac{\gamma^2}{2}}\\
\le & \e^{C\gamma^2}(4|D|)^{\frac{\gamma^2}{2}}2^{n\frac{\gamma^2}{2}}\\
= & \e^{(C+\log 2)\gamma^2} |D|^{\frac{\gamma^2}{2}}2^{n\frac{\gamma^2}{2}}.
\end{align*}
Taking $C_D=\e^{(C+\log 2)}$ gives the conclusion.
\end{proof}

\subsection{Lower local dimension estimates}

We will need the von Bahr-Esseen inequality  on $p$th moments of random variables for $1\le p \le 2$ and the Rosenthal inequality on $p$th moments of random variables for $p>2$.

\begin{theorem}\cite[Theorem 2]{vBE65}{\rm (von Bahr-Esseen)}.\label{thm2.1}
Let $\{X_m:1\le m \le n\}$ be a sequence of random variables satisfying
\[
\mathbb{E}\big(X_{m+1}\big|X_1+\ldots+X_m\big)=0, \quad 1\le m \le n-1.
\]
Then for $1\le p \le 2$
\[
\mathbb{E}\Big(\Big|\sum_{m=1}^n X_m\Big|^p\Big)\le 2\sum_{m=1}^n\mathbb{E}\big(|X_m|^p\big).
\]
\end{theorem}

\begin{theorem}\cite[Theorem 3]{Ro70}{\rm (Rosenthal)}.\label{thm2.2}
Let $\{X_m:1\le m \le n\}$ be a sequence of independent random variables with $\mathbb{E}(X_m)=0$ for $m=1,\ldots,n$. Then for $p>2$ there exists a constant $K_p$ such that
\[
\mathbb{E}\Big(\Big|\sum_{m=1}^n X_m\Big|^p\Big)\le K_p\max\Bigg\{\Big(\sum_{m=1}^n\mathbb{E}\big(|X_m|^2\big)\Big)^{p/2},\sum_{m=1}^n \mathbb{E}(|X_m|^p)\Bigg\}.
\]
\end{theorem}

The  following lemma bounds the difference of the total mass of the circle averages over consecutive radii $2^{-n}$.

\begin{lemma}\label{lemma1}
Let $\nu$ be a positive finite Borel measure on $D$ such that
\[
\nu(B(x,r))\le C r^\alpha
\]
for all $x\in\mathrm{supp}(\nu)$ and $r>0$. For $n\geq 1$, define the circle averages of the GFF on $\nu$ by
\begin{equation}\label{defnutiln1}
\widetilde{\nu}_{n}(dx)=2^{-n\gamma^2/2} \e^{\gamma \Gamma(\rho_{x,2^{-n}})} \, \nu(dx), \quad  x\in D.
\end{equation}
For $p\ge 1$ there exists a constant $0<C_p<\infty$ depending only on $D$, $\gamma$, $p$ such that for every Borel subset $A\subset D$ and for all integers $n\ge 1$,
\begin{equation}\label{estim1}
\mathbb{E}(|\widetilde{\nu}_{n+1}(A)-\widetilde{\nu}_{n}(A)|^p)\le C_p  |D|^{\frac{\gamma^2p^2}{2}}2^{-n(\alpha -\frac{\gamma^2}{2}p)(p-1)} \nu(A)
\end{equation}
if $1\le p \le 2$ and
\begin{equation}\label{estim2}
\mathbb{E}(|\widetilde{\nu}_{n+1}(A)-\widetilde{\nu}_{n}(A)|^p)\le C_p |D|^{\gamma^2p}2^{-n(\alpha -\gamma^2)\frac{p}{2}}\nu(A)^{\frac{p}{2}}+C_p |D|^{\frac{\gamma^2p^2}{2}} 2^{-n(\alpha -\frac{\gamma^2}{2}p)(p-1)}\nu(A)
\end{equation}
if $p>2$.
\end{lemma}

\begin{proof}
The proof follows the same lines as the proof of \cite[Proposition 3.1]{BJM10}. Fix a Borel subset $A\subset D$. For $S\in \mathcal{S}_n$ with $S\cap A\neq \emptyset$ recall that $\widetilde{S}$ is the $3$-fold enlargement of $S$ in $D$. By assumption (A0) we have that $\widetilde{S}$ is simply connected. Thus from \eqref{indep} we can write
\begin{equation}\label{decomp1}
\Gamma=\Gamma^{\widetilde{S}}+\Gamma_{\widetilde{S}},
\end{equation}
where $\Gamma^{\widetilde{S}}$ and $\Gamma_{\widetilde{S}}$ are two independent Gaussian processes on $\mathcal{M}$ with covariance functions $G_{\widetilde{S}}$ and $G_D-G_{\widetilde{S}}$ respectively. We can also choose a version of the process such that $\Gamma^{\widetilde{S}}$ vanishes on all measures supported in $D\setminus \widetilde{S}$, and $\Gamma_{\widetilde{S}}$ restricted to $\widetilde{S}$ is harmonic, that is for each measure $\rho$ supported in $\widetilde{S}$,
\[
\Gamma_{\widetilde{S}}(\rho)=\int_{\widetilde{S}} h_{\widetilde{S}}(x) \, \rho(dx),
\]
where $h_{\widetilde{S}}(x)=\Gamma(\tau_{\widetilde{S},x}), \ x\in \widetilde{S},$ is harmonic, where $\tau_{\widetilde{S},x}$ is the exit distribution of $\widetilde{S}$ by a Brownian motion started from $x$. In particular, by harmonicity,
\begin{equation}\label{decomp2}
\Gamma(\rho_{x,2^{-n}})=\Gamma^{\widetilde{S}}(\rho_{x,2^{-n}})+\Gamma(\tau_{\widetilde{S},x}), \ x \in S,
\end{equation}
where $\big\{\Gamma^{\widetilde{S}}(\rho_{x,2^{-n}}):x\in S\big\}$ and $\big\{\Gamma(\tau_{\widetilde{S},x}):x\in S\big\}$ are independent.

There is a universal integer $N$ such that the family $\mathcal{S}_n$ can be decomposed into $N$ subfamilies $\mathcal{S}_n^1,\ldots,\mathcal{S}_n^N$ such that for each $j=1,\ldots,N$, the closures of $\widetilde{S}$ and $\widetilde{S'}$ are disjoint for all $S,S'\in \mathcal{S}_n^j$. Let $\mathcal{S}_n^j(A)=\{S\in \mathcal{S}_n^j:S\cap A \neq \emptyset\}$ for $j=1,\ldots,N$. From \eqref{defnutiln1},
\begin{eqnarray}
\widetilde{\nu}_{n+1}(A)-\widetilde{\nu}_{n}(A)&=&\int_A \big( 2^{-(n+1)\frac{\gamma^2}{2}}\e^{\gamma \Gamma(\rho_{x,2^{-n-1}})}-2^{-n\frac{\gamma^2}{2}}\e^{\gamma \Gamma(\rho_{x,2^{-n}})} \, \big)\nu(dx) \nonumber\\
&=&\sum_{j=1}^N\sum_{S\in \mathcal{S}_n^j(A)}\int_{S\cap A} \big( 2^{-(n+1)\frac{\gamma^2}{2}}\e^{\gamma \Gamma(\rho_{x,2^{-n-1}})}-2^{-n\frac{\gamma^2}{2}}\e^{\gamma \Gamma(\rho_{x,2^{-n}})} \,\big)  \nu(dx) \nonumber\\
&=&\sum_{j=1}^N\sum_{S\in \mathcal{S}_n^j(A)}\int_{S\cap A} U_S(x)V_S(x) \, \nu(dx), \label{decomp3}
\end{eqnarray}
where
\[
U_S(x)=2^{-n\frac{\gamma^2}{2}} \e^{\gamma \Gamma(\tau_{\widetilde{S},x})}
\]
and
\[
V_S(x)=2^{-\frac{\gamma^2}{2}}\e^{\gamma \Gamma^{\widetilde{S}}(\rho_{x,2^{-n-1}})}-\e^{\gamma \Gamma^{\widetilde{S}}(\rho_{x,2^{-n}})}
\]
using \eqref{decomp2}. Since the families of regions $\big\{\mathcal{S}_n^j(A)\big\}_{j=1}^N$ are disjoint,  we may choose a version of the process such that the decompositions in \eqref{decomp1} and \eqref{decomp2} hold simultaneously for all $S\in \mathcal{S}_n^j$. Thus  $\big\{\{U_S(x):x\in S\}: S\in \mathcal{S}_n^j(A)\big\}$ and $\big\{\{V_S(x): x\in S\}: S\in \mathcal{S}_n^j(A)\big\}$ are independent for each $j=1,\ldots,N$,  and $\big\{\{V_S(x): x\in S\}: S\in \mathcal{S}_n^j(A)\big\}$ are mutually independent and centred. By first applying  H\"{o}lder's inequlity to the sum over $j$ in  \eqref{decomp3}, then taking conditional expectation with respect to $\big\{\{V_S(x): x\in S\}: S\in \mathcal{S}_n^j\big\}$, then applying the von Bahr-Esseen inequality, Theorem \ref{thm2.1}, and Rosenthal inequality, Theorem \ref{thm2.2},  and finally taking the expectation, we get for $1\le p\le 2$,
\begin{equation}\label{epth}
\mathbb{E}\big(|\widetilde{\nu}_{n+1}(A)-\widetilde{\nu}_{n}(A)|^p\big)\le 2N^{p-1}\sum_{j=1}^N\sum_{S\in \mathcal{S}_n^j(A)}\mathbb{E}\Big(\Big| \int_{S\cap A}  U_S(x)V_S(x)\, \nu(dx)\Big|^p \Big),
\end{equation}
and for $p>2$,
\begin{equation}\label{epth2}
\begin{split}
\mathbb{E}\big(|\widetilde{\nu}_{n+1}(A)-\widetilde{\nu}_{n}(A)|^p\big)
\le N^{p-1}K_p\sum_{j=1}^N& \Bigg[\Bigg(\sum_{S\in \mathcal{S}_n^j(A)}\mathbb{E}\bigg(\Big| \int_{S\cap A}  U_S(x)V_S(x)\, \nu(dx)\Big|^2\bigg)\Bigg)^{p/2}\\
&+\sum_{S\in \mathcal{S}_n^j(A)}\mathbb{E}\Big(\Big| \int_{S\cap A}  U_S(x)V_S(x)\, \nu(dx)\Big|^p \Big)\Bigg].
\end{split}
\end{equation}
To estimate these terms, we 
use H\"{o}lder's inequality,  \eqref{decomp2} and Lemma \ref{explemma} to get, for $x\in S$ and $p\geq 1$,
\begin{align}
\mathbb{E}\big(U_S&(x)^p|V_S(x)|^p\big)\nonumber\\
&=\mathbb{E}\left(2^{-n\frac{\gamma^2p}{2}} \e^{\gamma p\Gamma(\tau_{\widetilde{S},x})}\Big|2^{-\frac{\gamma^2}{2}}\e^{\gamma \Gamma^{\widetilde{S}}(\rho_{x,2^{-n-1}})}-\e^{\gamma \Gamma^{\widetilde{S}}(\rho_{x,2^{-n}})}\Big|^p\right)\nonumber\\
&\le 2^{p-1}\mathbb{E}\left(2^{-n\frac{\gamma^2p}{2}} \e^{\gamma p\Gamma(\tau_{\widetilde{S},x})}\Big(2^{-\frac{\gamma^2p}{2}}\e^{\gamma p \Gamma^{\widetilde{S}}(\rho_{x,2^{-n-1}})} + \e^{\gamma p \Gamma^{\widetilde{S}}(\rho_{x,2^{-n}})}\Big)\right)\nonumber\\
&= 2^{p-1}\mathbb{E}\left(2^{-(n+1)\frac{\gamma^2p}{2}} \e^{\gamma p\Gamma(\rho_{x,2^{-n-1}})}+ 2^{-n\frac{\gamma^2p}{2}} \e^{\gamma p \Gamma(\rho_{x,2^{-n}})}\right)\nonumber\\
&\le 2^{p-1}C_{D}^{\gamma^2 p^2}|D|^{\frac{\gamma^2p^2}{2}}\left(2^{(n+1)\frac{\gamma^2}{2}(p^2-p)} + 2^{n\frac{\gamma^2}{2}(p^2-p)} \right)\nonumber\\
&= C'_p |D|^{\frac{\gamma^2p^2}{2}}2^{n\frac{\gamma^2p}{2}(p-1)} \label{uniup}
\end{align}
where $C'_p=2^{p-1}C_{D}^{\gamma^2 p^2}(2^{\frac{\gamma^2}{2}(p^2-p)}+1)$ only depends on $D$, $p$ and $\gamma$. Hence using H\"{o}lder's inequality and Fubini's theorem,
\begin{align*}\label{epint}
\mathbb{E}\left(\Big| \int_{S\cap A}  U_S(x)V_S(x)\, \nu(dx)\Big|^p \right) 
&\le \nu(S\cap A)^{p-1} \int_{S\cap A} \mathbb{E}\big(U_S(x)^p|V_S(x)|^p\big) \, \nu(dx),\\
&\le  \nu(S\cap A)^{p-1}C_p'  |D|^{\frac{\gamma^2p^2}{2}}2^{n\frac{\gamma^2p}{2}(p-1)}\nu(S\cap A). 
\end{align*}
Summing over $S\in \mathcal{S}_n^j(A)$  and deducing from the main hypothesis \eqref{mesgrow} that 
\[
\nu(S\cap A)^{p-1} \le C^{p-1} |S|^{\alpha(p-1)}\le (C 2^{\alpha/2})^{p-1} 2^{-n\alpha (p-1)},
\]
gives
\begin{align}
\sum_{S\in \mathcal{S}_n^j(A)}\mathbb{E}\Bigg(\Big| \int_{S\cap A}  U_S(x)V_S(x)\, \nu(dx)\Big|^p \Bigg)
&\leq  \sum_{S\in \mathcal{S}_n^j(A)}C''_p    |D|^{\frac{\gamma^2p^2}{2}}  2^{-n\alpha (p-1)} 2^{n\frac{\gamma^2p}{2}(p-1)}\nu(S\cap A)\nonumber\\
&= \sum_{S\in \mathcal{S}_n^j(A)}C''_p  |D|^{\frac{\gamma^2p^2}{2}}2^{-n(\alpha -\frac{\gamma^2 p} {2})(p-1)}\nu(S\cap A).\label{estima}
\end{align}
Summing this over $j$ and combining with \eqref{epth}, immediately gives \eqref{estim1} for $1\le p \le 2$,
where $C_p''= 2(NC 2^{\alpha/2})^{p-1}C_p'$. 

When $p>2$, for the first term in \eqref{epth2} we substitute \eqref{estima} with $p=2$ and use H\"{o}lder's or Jensen's inequality to get 
\begin{align*}
\Bigg(\sum_{S\in \mathcal{S}_n^j(A)}\mathbb{E}\bigg(\Big| \int_{S\cap A}  U_S(x)V_S(x)&\, \nu(dx)\Big|^2\bigg)\Bigg)^{p/2}
\le\Bigg(\sum_{S\in \mathcal{S}_n^j(A)}C_2'' |D|^{2\gamma^2} 2^{-n(\alpha -\frac{\gamma^2}{2}2)(2-1)} \nu(S\cap A) \Bigg)^{p/2}\\
\le &\  \bigg(\sum_{S\in \mathcal{S}^j_n(A)}\nu(S\cap A)\bigg)^{\frac{p}{2}-1}\sum_{S\in \mathcal{S}_n^j(A)}(C_2'')^{\frac{p}{2}} |D|^{\gamma^2p} 2^{-n(\alpha -\gamma^2)\frac{p}{2}} \nu(S\cap A).
\end{align*}
Thus for $p>2$,
\[
\mathbb{E}\big(|\widetilde{\nu}_{n+1}(A)-\widetilde{\nu}_{n}(A)|^p\big)
\le N^{p-1}K_p\big((C_2'')^{\frac{p}{2}}  |D|^{\gamma^2p}2^{-n(\alpha -\gamma^2)\frac{p}{2}}\nu(A)^{\frac{p}{2}}+C_p'' |D|^{\frac{\gamma^2p^2}{2}}2^{-n(\alpha -\frac{\gamma^2}{2}p)(p-1)}\nu(A)\big).
\]
Then we get the conclusion by setting $C_p=C_p''$ for $1\le p \le 2$ and $C_p=N^{p-1}K_p\big((C_2'')^{\frac{p}{2}}+C_p''\big)$ for $p>2$.
\end{proof}

\begin{corollary}\label{ascon}
Let $\nu$ be a positive finite Borel measure on $D$ such that
\[
\nu(B(x,r))\le C r^\alpha
\]
for all $x\in\mathrm{supp}(\nu)$ and $r>0$.  For $n\geq 1$, define the circle averages of the GFF on $\nu$ by
$$
\widetilde{\nu}_{n}(dx)=2^{-n\gamma^2/2} \e^{\gamma \Gamma(\rho_{x,2^{-n}})} \, \nu(dx), \quad  x\in D.
$$
If $\gamma^2/2<\alpha$ then almost surely $\widetilde{\nu}_{n}$ converges weakly to a non-trivial limit measure $\widetilde{\nu}$.
\end{corollary}

\begin{proof}
Take $1<p\le 2$ such that $\alpha-\frac{\gamma^2}{2}p>0$. By Lemma \ref{lemma1}, for every Borel set $A\subset D$ and $n\ge 1$,
\[
\mathbb{E}(|\widetilde{\nu}_{n+1}(A)-\widetilde{\nu}_{n}(A)|^p)\le C_p 2^{-n(\alpha -\frac{\gamma^2}{2}p)(p-1)} \nu(A).
\]
By using the Borel-Cantelli lemma this implies that almost surely $\widetilde{\nu}_{n}(A)$ converges to a limit which we denote by $\widetilde{\nu}(A)$. By dominated convergence theorem we have $\mathbb{E}(\widetilde{\nu}(A))=\int_{A}R(x,D)^{\gamma^2/2} \nu(dx)$. Let $\mathcal{S}= \cup_{n\ge 1}\mathcal{S}_n$. Since $\mathcal{S}$ is countable, it follows that almost surely  $\widetilde{\nu}_{n}(S)$ converges to $\widetilde{\nu}(S)$ for all $S\in\mathcal{S}$. This implies that almost surely $\widetilde{\nu}$ defines a measure on $D$ and $\widetilde{\nu}_{n}$ converges weakly to $\widetilde{\nu}$.
\end{proof}

Next, we estimate moments of $\widetilde{\nu}(S)$ for $S\in \mathcal{S}_n^\circ=\{S\in \mathcal{S}_n:\widetilde{S}\subset D\}$, where $\widetilde{S}$ is given by \eqref{stilde}.

\begin{lemma}\label{lemma3}
Let $\nu$ be a positive Borel measure on $D$ such that $ \nu(B(x,r))\le C r^\alpha$
for $x\in\mathrm{supp}(\nu)$ and $r>0$. For $1<p<2$ such that $\alpha-\frac{\gamma^2}{2}p>0$ there exists a constant $C_p$ such that for $n\ge 1$ and for all $S\in \mathcal{S}_n^\circ$,
\[
\mathbb{E}\big(\widetilde{\nu}(S)^p\big)\le C_p 2^{-n(\alpha -\frac{\gamma^2}{2}p)(p-1)} \nu(S).
\]
\end{lemma}

\begin{proof}
Recall from \eqref{decomp2}, that
\begin{equation}\label{decomp2a}
\Gamma(\rho_{x,2^{-n}})=\Gamma^{\widetilde{S}}(\rho_{x,2^{-n}})+\Gamma(\tau_{\widetilde{S},x}), \ x \in S,
\end{equation}
for $S \in \mathcal{S}_n^\circ$, where $\{\Gamma^{\widetilde{S}}(\rho_{x,2^{-n}}):x\in S\}$ and $\{\Gamma(\tau_{\widetilde{S},x}):x\in S\}$ are independent. This implies
\[
\widetilde{\nu}(dx)=\e^{\gamma\Gamma(\tau_{\widetilde{S},x})}\, \widetilde{\nu}^{\widetilde{S}}(dx),\ x\in S,
\]
where $\widetilde{\nu}^{\widetilde{S}}$ is the GMC measure of $\nu |_{\widetilde{S}}$ obtained from $\Gamma^{\widetilde{S}}$ by Corollary \ref{ascon}. By H\"older's inequality and independence,
\begin{eqnarray}
\mathbb{E}(\widetilde{\nu}(S)^p) &=& \mathbb{E}\left(\Big(\int_{S}\e^{\gamma\Gamma(\tau_{\widetilde{S},x})}\, \widetilde{\nu}^{\widetilde{S}}(dx) \Big)^p\right) \nonumber\\
&\le& \mathbb{E}\left( \widetilde{\nu}^{\widetilde{S}}(S)^{p-1}\int_{S}\e^{p\gamma\Gamma(\tau_{\widetilde{S},x})}\, \widetilde{\nu}^{\widetilde{S}}(dx) \right)\nonumber\\
&=&  \mathbb{E}\left( \widetilde{\nu}^{\widetilde{S}}(S)^{p-1}\int_{S}\mathbb{E}\big(\e^{p\gamma\Gamma(\tau_{\widetilde{S},x})}\big)\, \widetilde{\nu}^{\widetilde{S}}(dx) \right) \nonumber\\
&\le& \max_{x\in S}\mathbb{E}\big(\e^{p\gamma\Gamma(\tau_{\widetilde{S},x})}\big) \mathbb{E}\big(\widetilde{\nu}^{\widetilde{S}}(S)^{p}\big).\label{nusp}
\end{eqnarray}
To estimate the first term of \eqref{nusp}, the decomposition \eqref{decomp2a}, independence, and \eqref{eeg} give
\[
\mathbb{E}\big(\e^{p\gamma\Gamma(\tau_{\widetilde{S},x})}\big)
=\bigg(\frac{R(x,D)}{R(x,\widetilde{S})}\bigg)^{\frac{\gamma^2p^2}{2}}.
\]
Recalling \eqref{confrad}, that
\begin{equation}\label{cr}
{\rm dist}(x, \partial D) \leq R(x,D) \leq 4\, {\rm dist}(x, \partial D),
\end{equation}
and noting that  $\mathrm{dist}(x,\partial\widetilde{S})\ge 2^{-n}$, gives
\begin{equation}\label{u1}
 \max_{x\in S}\mathbb{E}\big(\e^{p\gamma\Gamma(\tau_{\widetilde{S},x})}\big) \le (4|D|)^{\frac{\gamma^2p^2}{2}} 2^{n\frac{\gamma^2p^2}{2}}.
\end{equation}
For the second term in \eqref{nusp}, for $m\ge n$ write
\begin{equation}
\widetilde{\nu}_{m}^{\widetilde{S}}(S)= \int_{S} 2^{-m{\frac{\gamma^2}{2}}} \e^{\gamma\Gamma^{\widetilde{S}}(\rho_{x,2^{-m}})}\nu(dx).\label{yms}
\end{equation}
By Minkowski's inequality,
\begin{equation}\label{expsum}
\mathbb{E}\big(\widetilde{\nu}^{\widetilde{S}}(S)^{p}\big)^{\frac{1}{p}}\  \le\  \mathbb{E}\big(\widetilde{\nu}_{n}^{\widetilde{S}}(S)^p\big)^{\frac{1}{p}}+\sum_{m=n}^\infty \mathbb{E}\big(|\widetilde{\nu}_{m+1}^{\widetilde{S}}(S)-\widetilde{\nu}_{m}^{\widetilde{S}}(S)|^p\big)^{\frac{1}{p}}.
\end{equation}
To estimate the first term of \eqref{expsum}, we apply H\"older's inequality to \eqref{yms}, apply \eqref{eeg} and bound $\nu(S)$ using the main growth condition \eqref{mesgrow}, to get 
\begin{eqnarray*}
\mathbb{E}\big(\widetilde{\nu}^{\widetilde{S}}_n(S)^{p}\big)&\leq&2^{-n{\frac{\gamma^2p}{2}}}\nu(S)^{p-1}\mathbb{E}\bigg(  \int_{S} \e^{p\gamma\Gamma^{\widetilde{S}}(\rho_{x,2^{-n}})}\nu(dx) \bigg)\\
&\leq&2^{-n{\frac{\gamma^2p}{2}}}C^{(p-1)}|S|^{\alpha(p-1)}2^{n{\frac{\gamma^2p^2}{2}}}
 \int_{S} R(x,\widetilde{S})^{\frac{\gamma^2p^2}{2}} \, \nu(dx)\\
&\le&C_1^{p-1}2^{-n(\alpha -\frac{\gamma^2}{2}p)(p-1)} \int_{S} R(x,\widetilde{S})^{\frac{\gamma^2p^2}{2}} \, \nu(dx)\\
&\le&  C_1^{p-1} 2^{-n(\alpha -\frac{\gamma^2}{2}p)(p-1)}\max_{x\in S}R(x,\widetilde{S})^{\frac{\gamma^2p^2}{2}} \nu(S),
\end{eqnarray*}
where $C_1 = 2^{\alpha/2}C$.
For the summed terms in \eqref{expsum},  Lemma \ref{lemma1}, applied to the domain $\widetilde{S}$ instead of the domain $D$,  gives for $m\ge n$,
$$
\mathbb{E}\big(|\widetilde{\nu}_{m+1}^{\widetilde{S}}(S)-\widetilde{\nu}_{m}^{\widetilde{S}}(S)|^p\big)
\le C_p |\widetilde{S}|^{\frac{\gamma^2p^2}{2}} 2^{-m(\alpha -\frac{\gamma^2}{2}p)(p-1)} \nu(S),
$$
where $C_p= 2^p(NC2^{\alpha/2})^{p-1}\big(2^{\frac{\gamma^2}{2}(p^2-p)}+1\big)$. Thus, from \eqref{expsum}, and using the fact that $\max_{x\in S}R(x,\widetilde{S}) \le 4 |\widetilde{S}|$,
\begin{eqnarray*}
\mathbb{E}\big(\widetilde{\nu}^{\widetilde{S}}(S)^{p}\big)^{\frac{1}{p}} &\le&\mathbb{E}\big(\widetilde{\nu}^{\widetilde{S}}_n(S)^{p}\big)^{\frac{1}{p}}+ \sum_{m=n}^\infty \Big[C_p  |\widetilde{S}|^{\frac{\gamma^2p^2}{2}} 2^{-m(\alpha -\frac{\gamma^2}{2}p)(p-1)} \nu(S)\Big]^{\frac{1}{p}}\\
&\le& C_p' \Big[ |\widetilde{S}|^{\frac{\gamma^2p^2}{2}} 2^{-n(\alpha -\frac{\gamma^2}{2}p)(p-1)} \nu(S)\Big]^{\frac{1}{p}}
\end{eqnarray*}
where $C_p' = 2^{\gamma^2p}C_1^{(p-1)/p} + C_p^{1/p}\big/\big(1- 2^{-(\alpha -\frac{\gamma^2}{2}p)(p-1)/p}\big)$.
Noting that $|\widetilde{S}|< 2\sqrt{2}\cdot 2^{-n}$ and applying \eqref{cr} again, we deduce that
\begin{equation}\label{u2}
\mathbb{E}\big(\widetilde{\nu}^{\widetilde{S}}(S)^{p}\big)\le  C_p^{''} 2^{-n(\alpha -\frac{\gamma^2}{2}p)(p-1)} 2^{-n\frac{\gamma^2p^2}{2}} \nu(S),
\end{equation}
where $C_p^{''}=(C_p^{'})^p2^{\frac{3\gamma^2p^2}{2}}$. Incorporating estimates \eqref{u1} and \eqref{u2} in \eqref{nusp} we conclude that
\[
\mathbb{E}\big(\widetilde{\nu}(S)^p\big)\le C_p^{''} (4|D|)^{\frac{\gamma^2p^2}{2}} 2^{-n(\alpha -\frac{\gamma^2}{2}p)(p-1)} \nu(S),
\]
so $C_p = C_p^{''} (4|D|)^{\frac{\gamma^2p^2}{2}}$ in the statement of the lemma.
\end{proof}

We can now obtain the lower bound for the local dimensions.

\begin{proposition}\label{locdimleft}
Let $\nu$ be a positive finite Borel measure on $D$ such that
\[
\nu(B(x,r))\le C r^\alpha
\]
for all $x\in\mathrm{supp}(\nu)$ and $r>0$. Then, almost surely, for $\widetilde{\nu}$-a.e. $x$,
\[
\liminf_{r\to 0} \frac{\log \widetilde{\nu}(B(x,r))}{\log r}\ge \alpha-\frac{\gamma^2}{2}.
\]
\end{proposition}

\begin{proof}
For $S\in\mathcal{S}_n$ denote by $\mathcal{N}(S)$ the set of at most 9 $2^{-n}$-neighbor squares of $S$ in $\mathcal{S}_n$ (including $S$ itself), that is all $S'\in\mathcal{S}_n$ such that $\overline{S}\cap\overline{S'}\neq \emptyset$. For $\kappa>0$ define
\[
E_n(\kappa):=\left\{S\in \mathcal{S}_n^\circ: \max_{S'\in\mathcal{N}(S)}\widetilde{\nu}(S')>2^{-n\kappa}\right\}.
\]
Then for all $p>1$,
\begin{eqnarray*}
\widetilde{\nu}(E_n(\kappa))&=&\sum_{S\in\mathcal{S}_n^\circ} \mathbf{1}_{\left\{\max_{S'\in\mathcal{N}(S)}\widetilde{\nu}(S')>2^{-n\kappa}\right\}} \widetilde{\nu}(S)\\
&\le&\sum_{S\in\mathcal{S}_n^\circ}  \sum_{S'\in\mathcal{N}(S)}2^{n\kappa(p-1)} \widetilde{\nu}(S')^{(p-1)}\widetilde{\nu}(S)\\
&=&2^{n\kappa(p-1)}\sum_{S\in\mathcal{S}_n^\circ}  \sum_{S'\in\mathcal{N}(S)}\widetilde{\nu}(S')^{(p-1)}\widetilde{\nu}(S).
\end{eqnarray*}
By H\"older's inequality,
\begin{equation*}
\mathbb{E}\left(\widetilde{\nu}(E_n(\kappa))\right)\le
2^{n\kappa(p-1)}\sum_{S\in\mathcal{S}_n^\circ}  \sum_{S'\in\mathcal{N}(S)} \mathbb{E}\left(\widetilde{\nu}(S')^{p}\right)^{\frac{p-1}{p}}\mathbb{E}\left(\widetilde{\nu}(S)^p\right)^{\frac{1}{p}}.
\end{equation*}
Note that $\#\mathcal{N}(S)\le 9$. Using Lemma \ref{lemma3},
\begin{align}
\mathbb{E}\big(\widetilde{\nu}(E_n(\kappa))\big)\le& C_p2^{-n(\alpha -\frac{\gamma^2}{2}p -\kappa)(p-1)}\sum_{S\in\mathcal{S}_n^\circ}  \sum_{S'\in\mathcal{N}(S)} \nu(S')^{\frac{p-1}{p}}\nu(S)^{\frac{1}{p}} \nonumber\\
\le &C_p2^{-n(\alpha -\frac{\gamma^2}{2}p -\kappa)(p-1)}\sum_{S\in\mathcal{S}_n^\circ}  \sum_{S'\in\mathcal{N}(S)} \max_{S''\in\mathcal{N}(S)}\nu(S'')\nonumber\\
\le &C_p2^{-n(\alpha -\frac{\gamma^2}{2}p -\kappa)(p-1)}\sum_{S\in\mathcal{S}_n^\circ}  
9\sum_{S'\in\mathcal{N}(S)}\nu(S')\nonumber\\
\le& 81C_p2^{-n(\alpha -\frac{\gamma^2}{2}p -\kappa)(p-1)} \nu(D), \label{emubound}
\end{align}
where the third and fourth inequalities come from the fact that each square $S\in\mathcal{S}_n$ will be counted in the summation at most $9$ times.
 
For all $0<\kappa<\alpha -\frac{\gamma^2}{2}p$,  inequality \eqref{emubound} implies that
\[
\sum_{n\ge 1} \mathbb{E}\big(\widetilde{\nu}(E_n(\kappa))\big)<\infty.
\]
Seeing $E_n(\kappa)$ as events in the product probability space $\Omega\times D$ with respect to the Peyri\`ere measure
\[
\mathbb{Q}(A)=\frac{1}{\nu(D)}\int_{\Omega\times D} \mathbf{1}_{A}(\omega,x)\, \widetilde{\nu}(dx)\mathbb{P}(d\omega), \ A\in \mathcal{B}(\Omega\times D)
\]
and applying the Borel-Cantelli lemma we get that, almost surely,
\[
\widetilde{\nu}(B(x,2^{-n}))\leq \widetilde{\nu}(B_n(x)) \leq 9\cdot 2^{-n\kappa}
\]
for all sufficiently large $n$ for $\widetilde{\nu}$-almost all $x$ (since $\lim_{n\to \infty} \widetilde{\nu}(\cup_{S\in\mathcal{S}_n^\circ}S)=\widetilde{\nu}(D)$), where
\[
B_n(x)=\bigcup_{S'\in\mathcal{N}(S_n(x))} S',
\]
and $S_n(x)$ is the square in  $\mathcal{S}_n$ containing $x$. Thus, almost surely, for $\widetilde{\nu}$-almost all $x$,
\[
\liminf_{n\to \infty} \frac{1}{\log 2^{-n}}\log \widetilde{\nu}(B(x,2^{-n})) \ge \kappa,
\]
for all $\kappa<\alpha -\frac{\gamma^2}{2}p$, where we may take $p$ arbitrarily close to $1$.
\end{proof}

\subsection{Upper local dimension estimates}

Throughout the proofs we will assume that the circle average process is a version satisfying the following modification theorem, so in particular all the functions $x\mapsto 2^{-n\gamma^2/2} \e^{\gamma \Gamma(\rho_{x,2^{-n}})}$ that we integrate against are continuous.  
\begin{proposition}\label{cont}\cite[Proposition 2.1]{HMP10}
The circle average process
\[
F:D\times(0,1] \ni (x,\epsilon)\mapsto \Gamma(\rho_{x,\epsilon}) \in \mathbb{R}
\]
has a modification $\widetilde{F}$ such that for every $0<\eta<1/2$ and $\eta_1,\eta_2>0$ there exists $M=M(\eta,\eta_1,\eta_2)$ that is almost surely finite and  such that
\[
\big|\widetilde{F}(x,\epsilon_1)-\widetilde{F}(y,\epsilon_2)\big|\le M \left(\log \frac{1}{\epsilon_1}\right)^{\eta_1}\frac{|(x,\epsilon_1)-(y,\epsilon_2)|^{\eta}}{{\epsilon_1}^{\eta+\eta_2}}
\]
for all $x,y\in D$ and $\epsilon_1,\epsilon_2\in(0,1]$ with $1/2\le \epsilon_1/\epsilon_2 \le 2$.
\end{proposition}

Let $0<\eta<1/2$, $\eta_1,\eta_2>0$  and $M=M(\eta,\eta_1,\eta_2)$ be the random number given by Proposition \ref{cont}. 
For $\epsilon>0$ let $A_\epsilon=\{M\le \epsilon^{-1}\}$. 
Let $\mathcal{S}_n'$ be the collection of $S\in\mathcal{S}_n$ such that $B(x_S,2\cdot2^{-n\frac{\eta}{\eta+\eta_2}})\subset D$, where  $x_S$ is the center of $S$.

The upper local dimension bound depends on the following lemma.

\begin{lemma}\label{lemma4}
Let $\nu$ be a positive finite Borel measure on $D$ and let $\widetilde{\nu}$ be a GMC measure of $\nu$. There exist a constant $C'>0$ such that for all $\epsilon>0$, $p\in(0,1)$, $n\ge 1$ and $S\in \mathcal{S}_n'$,
\[
\mathbb{E}\big(\mathbf{1}_{A_\epsilon}\widetilde{\nu}(S)^p\big)\le C' e^{\gamma p 4(\log 2)^{\eta_1} \epsilon^{-1}  \big(n \frac{\eta}{\eta+\eta_2}\big)^{\eta_1}} 2^{n \frac{\eta}{\eta+\eta_2} \frac{\gamma^2}{2} p(p-1)} \nu(S)^p.
\]
\end{lemma}
\medskip

\begin{proof}
Fix $n\ge 1$ and $S\in\mathcal{S}_n'$. Let $x_0$ denote the center of $S$. Let $l=n\frac{\eta}{\eta+\eta_2}$. For brevity let $U=B(x_S,2\cdot 2^{-l})$ so that $B(x,2^{-l})\subset U$ for all $x\in S$. By \eqref{indep} we can write 
\begin{equation}\label{indp2}
\Gamma=\Gamma^{U}+\Gamma_U,
\end{equation}
where $\Gamma^{U}$ and $\Gamma_U$ are independent, and $\Gamma_U$ is the harmonic extension of $\Gamma|_{D\setminus U}$ to $U$. Note that for $m\ge l$, 
\begin{equation}\label{indp3}
\Gamma_U(\rho_{x,2^{-m}})=\Gamma(\tau_{U,x}),
\end{equation}
where $\tau_{U,x}$ is the exit distribution on $\partial U$. This gives
\begin{align*}
\widetilde{\nu}(S)=&\lim_{m\to \infty} \int_S 2^{-m\frac{\gamma^2}{2}}e^{\gamma \Gamma(\rho_{x,2^{-m}})} \, \nu(dx)\\
=&\lim_{m\to \infty} \int_S e^{\gamma \Gamma(\tau_{U,x})} 2^{-m\frac{\gamma^2}{2}}e^{\gamma \Gamma^U(\rho_{x,2^{-m}})} \, \nu(dx).
\end{align*}
By \eqref{indp2} and \eqref{indp3}, for all $x,y\in S$ and $m\ge l$,
\[
|\Gamma(\tau_{U,x})-\Gamma(\tau_{U,y})|\le |\Gamma(\rho_{x,2^{-m}})-\Gamma(\rho_{y,2^{-m}})|+|\Gamma^U(\rho_{x,2^{-m}})-\Gamma^U(\rho_{y,2^{-m}})|.
\]
We may apply Proposition \ref{cont} to $\Gamma^U$ to choose a version of the process such that the circle average process of $\Gamma^U$ has the same H\"older regularity as that of $\Gamma$. Moreover, as  $U$ is a ball, we may choose the same constant $M=M(\eta,\eta_1,\eta_2)$ in Proposition \ref{cont} for both for $\Gamma$ and $\Gamma^U$. This gives
\begin{align*}
|\Gamma(\tau_{U,x})-\Gamma(\tau_{U,y})|\le &2 M (\log 2^{l})^{\eta_1}\frac{|S|^\eta}{(2^{-l})^{\eta+\eta_2}}\\
\le & 4(\log 2)^{\eta_1} M l^{\eta_1}.
\end{align*}
Given $\epsilon>0$ recall that $A_\epsilon=\{M\le \epsilon^{-1}\}$. Then for all $x\in S$,
\begin{align*}
\mathbf{1}_{A_\epsilon}e^{\gamma \Gamma(\tau_{U,x})} \le e^{\gamma \Gamma(\tau_{U,x_S})+\gamma 4(\log 2)^{\eta_1} \epsilon^{-1} l^{\eta_1}}.
\end{align*}
Thus
\[
\mathbf{1}_{A_\epsilon}\widetilde{\nu}(S)^p\le e^{\gamma p 4(\log 2)^{\eta_1} \epsilon^{-1} l^{\eta_1}}e^{\gamma p\Gamma(\tau_{U,x_S})} \lim_{m\to \infty} \left(\int_S 2^{-m\frac{\gamma^2}{2}}e^{\gamma \Gamma^U(\rho_{x,2^{-m}})} \, \nu(dx)\right)^p
\]
By independence, Fatou's lemma, Jensen's inequality and Fubini's theorem
\begin{align*}
\mathbb{E}\left(\mathbf{1}_{A_\epsilon}\widetilde{\nu}(S)^p\right) \le  e^{\gamma p 4(\log 2)^{\eta_1} \epsilon^{-1} l^{\eta_1}}&\mathbb{E}\left(e^{\gamma p\Gamma(\tau_{U,x_S})}\right)\\
&\times \liminf_{m\to \infty} \left( \int_S2^{-m\frac{\gamma^2}{2}}\mathbb{E}\left(e^{\gamma \Gamma^U(\rho_{x,2^{-m}})}\right) \, \nu(dx)\right)^p.
\end{align*}
To estimate the first of these expectations, 
\[
\mathbb{E}\left(e^{\gamma p\Gamma(\tau_{U,x_S})}\right)=\bigg(\frac{R(x_S,D)}{R(x_S,U)}\bigg)^{\frac{\gamma^2p^2}{2}}\le (2|D|)^{\frac{\gamma^2p^2}{2}} 2^{l \frac{\gamma^2p^2}{2}},
\]
using \eqref{confrad} with $\mathrm{dist}(x_S,\partial D)\le |D|$ and $\mathrm{dist}(x_S, \partial U)\ge 2\cdot 2^{-l}$.
For the second expectation,
\[
\mathbb{E}\left(e^{\gamma \Gamma^U(\rho_{x,2^{-m}})}\right)=2^{m\frac{\gamma^2}{2}}R(x,U)^{\frac{\gamma^2}{2}};
\]
thus
\begin{align*}
 \int_S2^{-m\frac{\gamma^2}{2}}\mathbb{E}\left(e^{\gamma \Gamma^U(\rho_{x,2^{-m}})}\right) \, \nu(dx)=\int_S R(x,U)^{\frac{\gamma^2}{2}} \, \nu(dx) \le 16^{\frac{\gamma^2}{2}}\cdot 2^{-l\frac{\gamma^2}{2}} \nu(S),
\end{align*}
where we have used that $R(x,U)\le 4\mathrm{dist}(x,\partial U) \le 16\cdot 2^{-l}$. Gathering these two estimates together, we finally obtain
\begin{align*}
\mathbb{E}\left(\mathbf{1}_{A_\epsilon}\widetilde{\nu}(S)^p\right) \le C' e^{\gamma p 4(\log 2)^{\eta_1} \epsilon^{-1} l^{\eta_1}} 2^{n \frac{\eta}{\eta+\eta_2} \frac{\gamma^2}{2} p(p-1)} \nu(S)^p,
\end{align*}
where $C'=16^\frac{\gamma^2}{2}(2|D|\vee 1)^{\frac{\gamma^2}{2}}$.
\end{proof}

We can now complete the upper bound for the local dimensions.

\begin{proposition}\label{locdimright}
Let $\nu$ be a positive finite Borel measure on $D$ and let $\widetilde{\nu}$ be a GMC measure of $\nu$. If there exists a constant $\beta>\gamma^2/2$ such that
\begin{equation}\label{nulb}
 \nu(B(x,r))\ge C^{-1}r^{\beta}
\end{equation}
for all $x\in\mathrm{supp}(\nu)$ and $r>0$. Then, almost surely, for $\widetilde{\nu}$-a.e. $x$,
\[
\limsup_{r\to 0} \frac{\log \widetilde{\nu}(B(x,r))}{\log r} \le \beta-\frac{\gamma^2}{2}.
\]
\end{proposition}

\begin{proof}
For $\kappa>0$ define
\[
E_n'(\kappa):=\left\{S\in \mathcal{S}_n':\widetilde{\nu}(S)<2^{-n\kappa}\right\}.
\]
Then for $p\in(0,1)$,
\begin{eqnarray*}
\widetilde{\nu}(E_n'(\kappa))&=&\sum_{S\in\mathcal{S}_n'} \mathbf{1}_{\left\{\widetilde{\nu}(S)<2^{-n\kappa}\right\}} \widetilde{\nu}(S)\\
&\le&\sum_{S\in\mathcal{S}_n'}2^{-n\kappa(1-p)} \widetilde{\nu}(S)^{-(1-p)}\widetilde{\nu}(S)\\
&=&2^{-n\kappa(1-p)}\sum_{S\in\mathcal{S}_n'} \widetilde{\nu}(S)^{p}.
\end{eqnarray*}
From Lemma \ref{lemma4},
\begin{align}\label{1aenu}
\mathbb{E}\big(\mathbf{1}_{A_\epsilon}\widetilde{\nu}(E_n'(\kappa))\big)\le& C'2^{-n\kappa(1-p)}e^{\gamma p (4\log 2)^{\eta_1} \epsilon^{-1} \big(n \frac{\eta}{\eta+\eta_2}\big)^{\eta_1}} 2^{n \frac{\eta}{\eta+\eta_2} \frac{\gamma^2}{2} p(p-1)}\sum_{S\in\mathcal{S}_n'}  \nu(S)^p.
\end{align}
Recall that $\mathcal{N}(S)$ is the set of all neighborhood $2^{-n}$-squares of $S$, including $S$ itself. Then, using  \eqref{nulb},
\begin{align*}
\sum_{S\in\mathcal{S}_n'}  \nu(S)^p\  \le & \sum_{S\in\mathcal{S}_n'} \mathbf{1}_{\{\nu(S)>0\}} \bigg(\sum_{S'\in\mathcal{N}(S)} \nu(S')\bigg)^p\\
=& \sum_{S\in\mathcal{S}_n'}  \mathbf{1}_{\{\nu(S)>0\}}\bigg(\sum_{S'\in\mathcal{N}(S)} \nu(S')\bigg)^{p-1}  \bigg(\sum_{S'\in\mathcal{N}(S)} \nu(S')\bigg)\\
\le& \sum_{S\in\mathcal{S}_n'} \mathbf{1}_{\{\nu(S)>0\}} \left(\nu(B(x_S',2^{-n}))\right)^{p-1}\bigg(\sum_{S'\in\mathcal{N}(S)} \nu(S')\bigg)\\
\le& \sum_{S\in\mathcal{S}_n'}  \mathbf{1}_{\{\nu(S)>0\}}C^{1-p}2^{n\beta(1-p)}\sum_{S'\in\mathcal{N}(S)} \nu(S')\\
\le&\  9C^{1-p}2^{n\beta(1-p)}\nu(D),
\end{align*}
where $x_S'\in S\cap\mathrm{supp}(\nu)$ can be chosen arbitrarily. (Note that $S\cap\mathrm{supp}(\nu)\neq \emptyset$ since $\nu(S)>0$ and the last inequality comes from the fact that each square $S\in\mathcal{S}_n$ will be counted in the summation at most $9$ times.) Using this estimate in \eqref{1aenu},
\[
\mathbb{E}\big(\mathbf{1}_{A_\epsilon}\widetilde{\nu}(E_n'(\kappa))\big)\le 9C^{1-p}C'e^{\gamma p  4(\log 2)^{\eta_1} \epsilon^{-1} (n \frac{\eta}{\eta+\eta_2})^{\eta_1}} 2^{-n(1-p)\left(\kappa-(\beta-\frac{\eta}{\eta+\eta_2}\frac{\gamma^2}{2}p)\right)}.
\]
 
For all $\kappa>\beta-\frac{\eta}{\eta+\eta_2}\frac{\gamma^2}{2}p$ and $\eta_1<1$,  the above inequality implies that
\[
\sum_{n\ge 1} \mathbb{E}\big(\mathbf{1}_{A_\epsilon}\widetilde{\nu}(E_n'(\kappa))\big) <\infty.
\]
Seeing $\widetilde{\nu}(E_n'(\kappa))$ as events of the product probability space $\Omega\times D$ with respect to the measure
\[
\mathbb{Q}_\epsilon(A)=\frac{1}{\nu(D)\mathbb{P}(A_\epsilon)}\int_{\Omega\times D} \mathbf{1}_{A_\epsilon}(\omega)\mathbf{1}_A(\omega,x) \, \widetilde{\nu}(dx) \mathbb{P}(d\omega),
\]
and applying Borel-Cantelli lemma we obtain that, for $\mathbb{P}$-almost every $\omega\in A_\epsilon$,  the measure $\widetilde{\nu}(S_n(x)) \ge 2^{-n\kappa}$ for all sufficiently large $n$ for $\widetilde{\nu}$-almost all $x$ such that $S_n(x)\in\mathcal{S}_n'$, where  $S_n(x)$ is the dyadic square in $\mathcal{S}_n$ containing $x$. Note that $S_n(x)\subset B(x,2\cdot 2^{-n})$ and $\lim_{n\to \infty}\widetilde{\nu}\left(\cup_{S\in\mathcal{S}_n'} S\right)= \widetilde{\nu}(D)$. Thus, for $\mathbb{P}$-almost every $\omega\in A_\epsilon$, for $\widetilde{\nu}$-almost all $x$,
\[
\limsup_{n\to \infty} \frac{1}{\log 2^{-n}}\log \widetilde{\nu}(B(x,2\cdot 2^{-n})) \le \kappa,
\]
for all $\kappa>\beta-\frac{\eta}{\eta+\eta_2}\frac{\gamma^2}{2}p$. Since $p$ can be chosen arbitraily close to $1$ and $\eta_2$ arbitraily close to $0$, and $\mathbb{P}(\cup_{\epsilon>0} A_\epsilon)=1$, this gives the conclusion.
\end{proof}

Propositions \ref{locdimleft} and \ref{locdimright} combine to give Theorem \ref{locdim}. 

\medskip

\noindent {\it Proof of Corollary \ref{cored}.}  By Egorov's theorem for $\delta>0$ with  $\|\nu\|-\delta>0$ we can find a measurable set $E_\delta\subset D$ with $\nu(E_\delta)>\|\nu\|-\delta$ such that for $\epsilon>0$ with $\alpha-\epsilon-\frac{\gamma^2}{2}>0$ there exists a constant $0<C_{\delta,\epsilon}<\infty$ such that for all $x\in E_\delta$ and $r>0$,
\[
C_{\delta,\epsilon}^{-1} r^{\alpha+\epsilon} \le\nu(B(x,r))\le C_{\delta,\epsilon} r^{\alpha-\epsilon}.
\]
Write $\nu_\delta=\nu|_{E_{\delta}}$. Since $\nu_\delta(A)\le \nu(A)$ for every set $A$, the measure $\nu_\delta$ satisfies the assumptions of Proposition \ref{locdimleft} and Corollary \ref{ascon}. Therefore $\widetilde{\nu_\delta}$ is well-defined and non-trivial, and almost surely for $\widetilde{\nu_\delta}$ a.e. $x\in E_\delta$,
\[
\liminf_{r\to 0} \frac{\log \widetilde{\nu_\delta}(B(x,r))}{\log r}\ge \alpha-\epsilon-\frac{\gamma^2}{2}.
\]
For the upper bound, as in the proof of Proposition \ref{locdimright}, defining
\[
E_n'(\kappa):=\left\{S\in \mathcal{S}_n':\widetilde{\nu_\delta}(S)<2^{-n\kappa}\right\}
\]
we get for $p\in(0,1)$
\[
\mathbb{E}\big(\mathbf{1}_{A_\epsilon}\widetilde{\nu_\delta}(E_n'(\kappa))\big)\le C'2^{-n\kappa(1-p)}e^{\gamma p 4(\log 2)^{\eta_1} \epsilon^{-1} \big(n \frac{\eta}{\eta+\eta_2}\big)^{\eta_1}} 2^{n \frac{\eta}{\eta+\eta_2} \frac{\gamma^2}{2} p(p-1)}\sum_{S\in\mathcal{S}_n'}  \nu_\delta(S)^p.
\]
Then in the next step we can make the following alternative estimate:
\begin{align*}
\sum_{S\in\mathcal{S}_n'}  \nu_\delta(S)^p\  \le & \sum_{S\in\mathcal{S}_n'} \mathbf{1}_{\{\nu_\delta(S)>0\}} \bigg(\sum_{S'\in\mathcal{N}(S)} \nu_\delta(S')\bigg)^p\\
\le & \sum_{S\in\mathcal{S}_n'} \mathbf{1}_{\{\nu_\delta(S)>0\}} \bigg(\sum_{S'\in\mathcal{N}(S)} \nu(S')\bigg)^p\\
=& \sum_{S\in\mathcal{S}_n'}  \mathbf{1}_{\{\nu_\delta(S)>0\}}\bigg(\sum_{S'\in\mathcal{N}(S)} \nu(S')\bigg)^{p-1}  \bigg(\sum_{S'\in\mathcal{N}(S)} \nu(S')\bigg)\\
\le& \sum_{S\in\mathcal{S}_n'} \mathbf{1}_{\{\nu_\delta(S)>0\}} \left(\nu(B(x_S',2^{-n})\right)^{p-1}\bigg(\sum_{S'\in\mathcal{N}(S)} \nu(S')\bigg)\\
\le& \sum_{S\in\mathcal{S}_n'}  \mathbf{1}_{\{\nu_\delta(S)>0\}}C_{\delta,\epsilon}^{1-p}2^{n(\alpha+\epsilon)(1-p)}\sum_{S'\in\mathcal{N}(S)} \nu(S')\\
\le&\  9C_{\delta,\epsilon}^{1-p}2^{n(\alpha+\epsilon)(1-p)}\nu(D),
\end{align*}
where $x_S'$ is a point in $E_\delta\cap S$ since $\nu_\delta(S)>0$. Then following the same lines as in the proof of Proposition \ref{locdimright} we get that almost surely for $\widetilde{\nu_\delta}$ a.e. $x\in E_\delta$,
\[
\limsup_{r\to 0} \frac{\log \widetilde{\nu_\delta}(B(x,r))}{\log r}\le \alpha+\epsilon-\frac{\gamma^2}{2}.
\]
By taking a countable sequences $\epsilon_n\to 0$ we get that almost surely for $\widetilde{\nu_\delta}$ a.e. $x\in E_\delta$,
\[
\lim_{r\to 0} \frac{\log \widetilde{\nu_\delta}(B(x,r))}{\log r}= \alpha-\frac{\gamma^2}{2}.
\]

Now, since we may choose a decreasing sequence $\delta_n\to 0$ with the sequence of sets $E_{\delta_n}$ increasing, the limit $\lim_{n\to\infty} \widetilde{\nu_{\delta_n}}(S):=\widetilde{\nu}(S)$ exists for all $S\in\mathcal{S}$. This defines a random measure on $D$ and by monotone convergence theorem we have for every measurable set $A\subset D$,
\[
\mathbb{E}(\widetilde{\nu}(A))=\int_A R(x,D)^{\gamma^2/2} \nu(dx).
\]
Finally, from measure differential theory (see \cite[Theorem 2.14]{Mat95} for example), for $n\ge 1$, almost surely for $\widetilde{\nu}$-a.e. $x\in E_{\delta_n}$,
\[
\lim_{r\to 0} \frac{\widetilde{\nu_{\delta_n}}(B(x,r))}{\widetilde{\nu}(B(x,r)}=1
\]
and therefore
\[
\lim_{r\to 0} \frac{\log \widetilde{\nu}(B(x,r))}{\log r}= \alpha-\frac{\gamma^2}{2}.
\]
This yields the conclusion.
$\Box$

\section{Proof of Theorem \ref{thm1.1}}\label{sec4}

Throughout this section we will assume that the domain $D$ satisfies (A0) and the space  $\mathcal{T}$ and the parameterised family of measures $\nu_t = \nu\circ f_t^{-1}, t \in \mathcal{T}$ satisfy (A1), (A2) and (A3) of Section \ref{assumps}. For each $t\in \mathcal{T}$ we define the circle averages of the GFF $\Gamma$ with radius $2^{-n}$ on $\nu_t$ by
$$\widetilde{\nu}_{t,n}(dx)=2^{-n\gamma^2/2} \e^{\gamma \Gamma(\rho_{x,2^{-n}})} \, \nu_t(dx), \quad  x\in D,
$$
and the total mass of $\widetilde{\nu}_{t,n}$ by
$$
Y_{t,n}:=\|\widetilde{\nu}_{t,n}\|.
$$
By (A1) and Corollary \ref{ascon} almost surely the weak limit $\widetilde{\nu}_{t}= \mbox{\rm w-}\!\lim_{n\to\infty} \widetilde{\nu}_{t,n}$ exists and we let $Y_{t}=\|\widetilde{\nu}_{t}\|$ be its total mass. 

The proof of Proposition \ref{prop1}, from which Theorem \ref{thm1.1} follows easily, depends on two lemmas: Lemma \ref{lemma1a} concerns the expected convergence speed of $Y_{t,n}$ as $n\to \infty$ and Lemma \ref{lemma2} gives a stochastic equicontinuity condition on $Y_{t,n}$ in $t$. 
These lemmas are combined in an inductive manner reminiscent of the proof of the Kolmogorov-Chentsov theorem.

For $p\ge 1$ define
\begin{equation}\label{defsp}
s_{\alpha_1,\gamma}(p)=\left\{\begin{array}{ll}(\alpha_1-\frac{\gamma^2}{2}p)(p-1) & \text{if } 1\le p \le 2; \\ \min\Big\{(\alpha_1-\gamma^2)\frac{p}{2}
, (\alpha_1 -\frac{\gamma^2}{2}p)(p-1)\Big\} & \text{if } p>2. \end{array}\right.
\end{equation}

\begin{lemma}\label{lemma1a}
For $p\ge 1$ there exists a constant $0<C_p<\infty$ depending only on $D$, $p$ and $\gamma$ such that for all $t\in \mathcal{T}$ and $n\ge 1$,
\begin{equation}\label{estim1a}
\mathbb{E}(|Y_{t,n+1}-Y_{t,n}|^p)\le C_p 2^{-n s_{\alpha_1,\gamma}(p)}.
\end{equation}
\end{lemma}

\begin{proof}This is immediate by applying Lemma \ref{lemma1} to the circle averages of the GFF on the measures $\nu_t$ for all $t\in \mathcal{T}$, noting that $ \nu_t(D)\le \max_{t\in \mathcal{T}} \nu_t(D)$ and renaming $C_p\|\nu\|$ as $C_p$ when $1\le p \le 2$ and $C_p(\|\nu\|^{\frac{p}{2}}+\|\nu\|)$ as $C_p$ when $p>2$. 
\end{proof}

Recall the notation $M=M(\eta,\eta_1,\eta_2)$ from Proposition \ref{cont} and $A_\epsilon=\{M\le \epsilon^{-1}\}$ for $\epsilon>0$.

\begin{lemma}\label{lemma2}
For $q>1$ and $0<\eta<1/2$ there exists a constant $0<C_{q,\eta,\epsilon}<\infty$ such that for all $0<r<r_2$ and $s,t\in \mathcal{T}$ with $d(s,t)\le r$ and all $n\ge 1$,
\begin{equation}\label{estim2a}
\mathbb{E}\left(\mathbf{1}_{A_\epsilon}\max_{1\le m \le n}|Y_{s,m}-Y_{t,m}|^q \right) \le C_{q,\eta,\epsilon} r^{q((\eta\alpha_2)\wedge \alpha_2')} 2^{nq(\frac{1}{2}+\frac{\gamma^2}{2}(q-1))}.
\end{equation}
\end{lemma}

\begin{proof}
For $x\in D$ and $m\ge 1$ let
\[
\overline{F}_m(x)=\gamma \Gamma(\rho_{x,2^{-m}})-\frac{\gamma^2}{2} m\log 2.
\]
By (A0) and Lemma \ref{explemma} we have
\[
\mathbb{E}(e^{q\overline{F}_m(x)}) \le C_{D,\gamma q} 2^{m\frac{\gamma^2q}{2}(q-1)}
\]
and therefore
\begin{align}
\mathbb{E}\bigg(\max_{1\le m \le n} e^{q\overline{F}_m(x)}\bigg) &\le \mathbb{E}\bigg(\sum_{m=1}^n e^{q\overline{F}_m(x)}\bigg)\nonumber\\
&= \sum_{m=1}^n \mathbb{E}\left(e^{q\overline{F}_m(x)}\right)\nonumber\\
&\leq C_{D,\gamma q}' 2^{n\frac{\gamma^2q}{2}(q-1)} \label{max}
\end{align}
where $C_{D,\gamma q}'$ only depends on $D$, $\gamma$ and $q$. 

For $s,t\in \mathcal{T}$ with $d(s,t)\le r\le r_2$, (A2) implies
\begin{equation}\label{A21}
 \sup_{u\in I_s\cap I_{t}}|f_s(u)-f_t(u)| \le C_2r^{\alpha_2}
\end{equation}
and
\begin{equation}\label{A22}
\max\big\{\nu(I_s\setminus I_t),\nu(I_t\setminus I_s)\big\}\ \leq\  \nu(I_s\Delta I_t)\ \le \ C_2r^{\alpha_2'}.
\end{equation}
We need to estimate the difference between
\[
Y_{s,m}=\int_{I_s} \e^{\overline{F}_m(f_s(u))} \, \nu(du)
\quad \mbox{ and }\quad
Y_{t,m}=\int_{I_{t}} \e^{\overline{F}_m(f_t(u))} \, \nu(du).
\]
For $u\in I_s\cap I_{t}$ and $m\ge 1$ let $t_{u,m}\in \overline{B_{d}(t,r)}$ be such that
\[
\overline{F}_m(f_{t_{u,m}}( u))=\inf_{s\in B_{d}(t,r)} \overline{F}_m(f_{s}(u)).
\]
Define
\[
Y_{s,m}^*=\int_{I_s\cap I_t} \e^{\overline{F}_m(f_s( u))} \, \nu(du), \quad Y_{t,m}^*=\int_{I_{s}\cap I_t} \e^{\overline{F}_m(f_t( u))} \, \nu(du)
\]
and
\[
Y_{m}^*=\int_{I_s\cap I_{t}} \e^{\overline{F}_m(f_{t_{u,m}}( u))} \,  \nu(du).
\]
Then
\begin{eqnarray}
|Y_{s,m}-Y_{t,m}|&\le & \int_{I_s\setminus I_t} \e^{\overline{F}_m(f_s( u))} \,  \nu(du) +\int_{I_t\setminus I_s} \e^{\overline{F}_m(f_t( u))} \,  \nu(du)\nonumber\\
&& + |Y_{s,m}^*-Y_{m}^*|+|Y_{t,m}^*-Y_{m}^*|. \label{diffest}
\end{eqnarray}
Firstly, using Jensen's inequality, Fubini's theorem, \eqref{max} and \eqref{A22},
\begin{eqnarray}
\mathbb{E}\left(\max_{1\le m \le n}\Big(\int_{I_s\setminus I_t} \e^{\overline{F}_m(f_s( u))} \, \nu(du)\Big)^q\right) 
&\le & C_{D,\gamma q}'  2^{n\frac{\gamma^2}{2}(q^2-q)} \nu(I_s\setminus I_t)^q \nonumber\\
&\le & C_{D,\gamma q}'C_2^q r^{q\alpha_2'} 2^{n\frac{\gamma^2q}{2}(q-1)},\label{ydelta2'}
\end{eqnarray}
and similarly 
\begin{equation}\label{ydelta2''}
\mathbb{E}\left(\max_{1\le m \le n}\Big(\int_{I_t\setminus I_s} \e^{\overline{F}_m(f_t(  u))} \,  \nu(du)\Big)^q\right)\ \ \  \le \ \ \ C_{D,\gamma q}'C_2^q  r^{q\alpha_2'} 2^{n\frac{\gamma^2q}{2}(q-1)}.
\end{equation}
Secondly, 
\begin{eqnarray*}
\left|Y_{s,m}^*-Y_{m}^*\right| &=& \int_{I_s\cap I_t} \left( \e^{\overline{F}_m(f_s( u))}-\e^{\overline{F}_m(f_{ t_{u,m}}( u))}\right) \, \nu(du)\\
&=&\int_{I_s\cap I_t} \e^{\overline{F}_m(f_s( u))} \left(1 -\e^{-(\overline{F}_m(f_s( u))-\overline{F}_m(f_{ t_{u,m}}( u)))}\right) \, \nu(du).
\end{eqnarray*}
From \eqref{A21}, $0\le f_s( u)-f_{ t_{u,m}}( u)\le C_2 r^{\alpha_2}$, so by Proposition \ref{cont}, given $0<\eta<1/2$ and $\eta_1,\eta_2>0$, we can find random constants $M\equiv M(\eta,\eta_1,\eta_2)$ such that
\[
\overline{F}_m(f_s(u))-\overline{F}_m(f_{t_{u,m}}( u))\le C_2^\eta M (m\log2)^{\eta_1}2^{m(\eta+\eta_2)}r^{\eta\alpha_2}.
\]
Since $1-\e^{-x}\le x$,
\[
\left|Y_{s,m}^*-Y_{m}^*\right|\le C_2^\eta M (m\log2)^{\eta_1}2^{m(\eta+\eta_2)}r^{\eta\alpha_2} Y_{s,m}^*.
\]
For $\epsilon>0$ recall that $A_\epsilon=\{M\le \epsilon^{-1}\}$ is the event that $M$ is bounded by $\epsilon^{-1}$. Then
\[
\mathbf{1}_{A_\epsilon}\left|Y_{s,m}^*-Y_{m}^*\right|\le C_2^\eta \epsilon^{-1} (m\log2)^{\eta_1}2^{m(\eta+\eta_2)}r^{\eta\alpha_2} Y_{s,m}^*.
\]
By using similar estimates to \eqref{ydelta2'} and \eqref{ydelta2''} for $Y_{s,m}^*$ and $Y_{t,m}^*$we get
\begin{equation}\label{ydelta1'}
\mathbb{E}\left(\mathbf{1}_{A_\epsilon}\max_{1\le m \le n}\left|Y_{s,m}^*-Y_{m}^*\right|^q\right) \le C_{q,\eta,\epsilon}'(n\log2)^{q\eta_1}2^{qn(\eta+\eta_2)}r^{q\eta \alpha_2 } 2^{n\frac{\gamma^2q}{2}(q-1)},
\end{equation}
and 
\begin{equation}\label{ydelta1''}
\mathbb{E}\left(\mathbf{1}_{A_\epsilon}\max_{1\le m \le n}\left|Y_{t,m}^*-Y_{m}^*\right|^q\right) \le C_{q,\eta,\epsilon}'(n\log2)^{q\eta_1}2^{qn(\eta+\eta_2)}r^{q\eta\alpha_2 } 2^{n\frac{\gamma^2q}{2}(q-1)}.
\end{equation}
where $C_{q,\eta,\epsilon}'=C_{D,\gamma q}'C_2^{q\eta} \epsilon^{-q}\|\nu\|^q$. Finally, using H\"{o}lder's inequality in \eqref{diffest} and incorporating \eqref{ydelta2'}, \eqref{ydelta2''} \eqref{ydelta1'} and \eqref{ydelta1''}, 
\begin{align*}
\mathbb{E}\Big(\mathbf{1}_{A_\epsilon}\max_{1\le m \le n}|Y_{s,m}-Y_{t,m}|^q \Big) \le & 4^{q-1}\Big(2C_{D,\gamma q}' C_2^q r^{q\alpha_2'}2^{n\frac{\gamma^2q}{2}(q-1)}\\
&+ 2C_{q,\eta,\epsilon}' (n\log2)^{q\eta_1}2^{qn(\eta+\eta_2)}r^{q\eta\alpha_2 } 2^{n\frac{\gamma^2q}{2}(q-1)}\Big),
\end{align*}
so by taking $\eta_1$, $\eta_2$ close to $0$ there exists a constant $C_{q,\eta,\epsilon}$ such that
\[
\mathbb{E}\Big(\mathbf{1}_{A_\epsilon}\max_{1\le m \le n}|Y_{s,m}-Y_{t,m}|^q \Big) \le C_{q,\eta,\epsilon} r^{q((\eta\alpha_2)\wedge \alpha_2' )} 2^{nq(\frac{1}{2}+\frac{\gamma^2}{2}(q-1))}.
\]
\end{proof}

The next proposition combines Lemmas  \ref{lemma1a} and \ref{lemma2} to obtain an estimate for the continuity exponent of $Y_{t,n}$ in $t$ that is uniform in $n$, from which Theorem \ref{thm1.1} will follow easily.  In the latter part of the proof we collect together the various estimates used to obtain a value for $\beta_0$.
Recall  the definitions of $s_{\alpha_1,\gamma}(p)$ from  \eqref{defsp}.

\begin{proposition}\label{prop1}
If there exist $p>1$ and $q>1$ such that
\begin{equation}\label{holderest1}
0<\frac{\frac{1}{2}+\frac{\gamma^2}{2}(q-1)}{(\frac{1}{2}\alpha_2)\wedge \alpha_2'-\frac{k}{q}}<\frac{s_{\alpha_1,\gamma}(p)}{k}<\infty
\end{equation}
then there are numbers $C,\beta>0$ such that, almost surely, there exists a (random) integer $N$ such that for all $s,t\in \mathcal{T}$ with $d(s,t)\le 2^{-N}$,
\begin{equation}\label{holderest3}
\sup_{n\ge 1} |Y_{s,n}-Y_{t,n}| \le C d(s,t)^\beta.
\end{equation}
\end{proposition}

\begin{proof}\label{prprop1}

By (A3), without loss of generality, we can view $\mathcal{T}$ itself as a convex subset of $[0,1]^k$. For $n\ge 1$ write
\[
\mathcal{T}_n=\{(i_12^{-n},\cdots,i_k2^{-n})\in \mathcal{T}: i_1,\ldots,i_k\in\{0,\ldots,2^{-n}\}\}
\]
Note that $\#\mathcal{T}_n\le 2^{nk}$. Given $p>1$ and $q>1$ such that \eqref{holderest1} holds, choose positive integers $\ell$ and $\zeta$ such that
\begin{equation}\label{gap}
\frac{\frac{1}{2}+\frac{\gamma^2}{2}(q-1)}{(\frac{1}{2}\alpha_2)\wedge \alpha_2'-\frac{k}{q}}<\frac{\zeta}{\ell}<\frac{s_{\alpha_1,\gamma}(p)}{k}.
\end{equation}
Write $\eta'=(\eta\alpha_2)\wedge \alpha_2'$. If $0<\eta<1/2$ is close enough to $1/2$ then
\begin{align*}
\ell s_{\alpha_1,\gamma}(p)-\zeta k:=&\ \delta_1>0,\\
\zeta(\eta' q -k)-q\ell \big(\frac{1}{2}+\frac{\gamma^2}{2}(q-1)\big):=&\ \delta_2>0.
\end{align*}
From Lemma \ref{lemma1a}, for $j=0,\ldots,\ell-1$,
\begin{align}
\mathbb{E}\Big(\max_{t\in \mathcal{T}_{n\zeta}} & |Y_{t,j+(n+1)\ell}-Y_{t,j+n\ell}|^p \Big) \nonumber\\
& \le \sum_{t\in \mathcal{T}_{n\zeta}}\ell^{p-1}\sum_{k=0}^{\ell-1}\mathbb{E}\left(|Y_{t,j+n\ell+k+1}-Y_{t,j+n\ell+k}|^p \right) \nonumber\\
& \le  2^{n\zeta k}\ell^{p-1}C_p\sum_{k=0}^{\ell-1} 2^{-(j+n\ell+k)(s_{\alpha_1,\gamma}(p)\wedge e_{\alpha_1,\alpha_3,\gamma}(p))}\nonumber\\
&\le C 2^{-n\delta_1},\label{est1}
\end{align}
where $C=\ell^{p-1}C_p\big(1-2^{-s_{\alpha_1,\gamma}(p)\wedge e_{\alpha_1,\alpha_2,\gamma}(p)}\big)^{-1}$.

For $n\ge 1$ let 
$$\mathcal{P}_n^\zeta=\big\{(s,t)\in \mathcal{T}_n\times\mathcal{T}_n: d(s,t)\le 2^\zeta  \sqrt{k}  2^{-n}\big\}.$$
Note that $\# \mathcal{P}_n^\zeta \le k16^\zeta  2^{nk}$.

Given $\epsilon>0$, by  Lemma \ref{lemma2}, for $n\ge 1$ satisfying $2^\zeta  \sqrt{k}  2^{-n\zeta}\le r_2$ and taking $r =2^\zeta  \sqrt{k}  2^{-n\zeta}$ in \eqref{estim2a},
 \begin{align}
\mathbb{E}\Big(\mathbf{1}_{A_\epsilon}\max_{(s,t)\in \mathcal{P}_{ n \zeta}}\max_{1\le m \le n\ell} & |Y_{s,m}-Y_{t,m}|^q \Big) \nonumber \\
&\leq k16^\zeta 2^{n\zeta k}C_{q,\eta,\epsilon} (2^\zeta \sqrt{k})^{q\eta'}2^{- n \zeta q\eta'} 2^{n\ell q(\frac{1}{2}+\frac{\gamma^2}{2}(q-1))}\nonumber \\
&\le C'2^{-n \delta_2 },\label{est2}
\end{align}
where $C'=k16^\zeta C_{q,\eta,\epsilon} (2^\zeta \sqrt{k})^{q\eta'}$. 

Choose $\beta>0$ such that both $\delta_1-\beta p>0$ and $\delta_2-\beta q>0$. Using Markov's inequality and \eqref{est1} and \eqref{est2},
\[
\mathbb{P}\left(\max_{j=0,\ldots,\ell-1}\max_{t\in \mathcal{T}_{n\zeta}} |Y_{t,j+(n+1)\ell}-Y_{t,j+n\ell}| > 2^{-n\beta}\right) \le \ell C 2^{-n(\delta_1-\beta p)}
\]
and
\[
\mathbb{P}\left( \mathbf{1}_{A_\epsilon}\max_{(s,t)\in \mathcal{P}_{n\zeta}} \max_{1\le m \le n\ell} |Y_{s,m}-Y_{t,m}| > 2^{-n\beta}\right) \le C' 2^{-n(\delta_2-\beta q)},
\]
provided $2^\zeta \sqrt{k}\, 2^{-n\zeta}\le r_2$. By the Borel-Cantelli lemma, for $\mathbb{P}$-almost every $\omega\in A_\epsilon$ there exists a random integer $N$ with $2^\zeta  \sqrt{k}\,  2^{-N\zeta}\le r_2$ such that, for all $n\ge N$, both
\begin{equation}\label{upper1}
\max_{j=0,\ldots,\ell-1}\max_{t\in \mathcal{T}_{n}} |Y_{t,j+(n+1)\ell}-Y_{t,j+n\ell}| \le 2^{-n\beta} 
\end{equation}
and
\begin{equation}\label{upper2}
\max_{(s,t)\in \mathcal{P}_{n\zeta }} \max_{1\le m \le n\ell} |Y_{s,m}-Y_{t,m}| \le 2^{-n\beta}.
\end{equation}

Fixing such an $N$ and $n\ge N+1$, as well as $j\in \{0,\ldots,\ell-1\}$, we will prove by induction on $M$ that for all $M\ge n$, and all $s,t\in \mathcal{T}_{M\zeta}$ with $d(s,t)\le 2^\zeta  \sqrt{k}\,  2^{-n\zeta}$,
\begin{equation}\label{induction}
\max_{0\le m \le M-1}|Y_{s,j+m\ell}-Y_{t,j+m\ell}|\le 2^{-n\beta}+2 \sum_{m=n}^{M-1}  (2^{-(m+1)\beta}+2^{-m \beta}).
\end{equation}

To start the induction, if $s,t \in \mathcal{T}_{ n \zeta}$ with $d(s,t) \leq 2^\zeta  \sqrt{k}\,  2^{-n\zeta}$, then 
$(s,t) \in \mathcal{P}_{n \zeta}$, so by \eqref{upper2},
$$\max_{0\le m \le n-1}|Y_{s,j+m\ell}-Y_{t,j+m\ell}| \le 2^{-n\beta}, $$
 which is \eqref{induction} when $M=n$ (with the summation null).

Now suppose that \eqref{induction} holds for some $M\geq n$. Let $s,t\in  \mathcal{T}_{(M+1)\zeta}$ with $d(s,t)\le2^\zeta  \sqrt{k}\,  2^{-n\zeta}$. Note that either $s\in\mathcal{T}_{M\zeta}$ or there exists an $s_*\in \mathcal{T}_{M\zeta}$ with $d(s,s_*)\le \sqrt{k}\,  2^{-M\zeta}=2^\zeta\sqrt{k}\, 2^{-(M+1)\zeta}$, same is true for $t$. Either way there are $s_*,t_*\in  \mathcal{T}_{M\zeta}$ with $d(s,s_*)\le 2^\zeta \sqrt{k}\, 2^{-(M+1) \zeta} \le 2^\zeta \sqrt{k}\,  2^{-n\zeta}$ and $d(t,t_*) \le 2^\zeta \sqrt{k}\, 2^{-(M+1) \zeta} \le 2^\zeta \sqrt{k}\, 2^{-n\zeta}$. Furthermore, since we assume that $\mathcal{T}$ is convex, we may choose $s_*,t_*$ such that $d(s_*,t_*)\le d(s,t)\le  2^\zeta  \sqrt{k}\,  2^{- n \zeta}$. Thus 
$(s,s_*),(t,t_*) \in \mathcal{P}_{(M+1) \zeta}$ and $(s_*,t_*)\in \mathcal{P}_{M\zeta}$. This gives, by considering the cases $1 \leq m \leq M-1$ and $m=M$ in the maximum separately, for all $j \in \{0,\ldots,l-1\}$,
\begin{align*}
\max_{0\le m \le M}&|Y_{s,j+m\ell}-Y_{t,j+m\ell}|\\
 \le &  \max_{0\le m \le M-1}|Y_{s_*,j+m\ell}-Y_{t_*,j+m\ell}| \\
&+ \max_{0\le m \le M}|Y_{s,j+m\ell}-Y_{s_*,j+m\ell}|  +\max_{0\le m \le M}|Y_{t_*,j+m\ell}-Y_{t,j+m\ell}| \\
&+ |Y_{s_*,j+M\ell}-Y_{s_*,j+(M-1)\ell}| + |Y_{t_*,j+M\ell}-Y_{t_*,j+(M-1)\ell}| \\
\le &\ 2^{-n\beta}+2 \sum_{m=n}^{M-1}  (2^{-(m+1)\beta}+2^{-m\beta}) +2  2^{-(M+1)\beta }+ 2 2^{-M\beta},
\end{align*}
using \eqref{upper2} and  \eqref{upper1}.
Thus \eqref{induction} is true with $M$ replaced by $M+1$, completing the induction.

Letting $M\to \infty$ in \eqref{induction} and summing the geometric series we get that for all $s,t\in \mathcal{T}_*=\bigcup_{n\ge 1} \mathcal{T}_n$ with $d(s,t)\le 2^\zeta\sqrt{k}\, 2^{-n \zeta }$,
\begin{equation*}
\sup_{m\ge 1}|Y_{s,m}-Y_{t,m}| \le C'' 2^{-n\beta},
\end{equation*}
where   $C''$ depends only on $\ell$ and $\beta$.

For $s,t\in \mathcal{T}_*$ with $d(s,t)\le 2^{-(N+1) \zeta}$ there exists a least $n\ge N+1$ such that $\sqrt{k}\, 2^{-(n+1) \zeta}< d(s,t)\le \sqrt{k}\, 2^{- n \zeta}\le 2^\zeta  \sqrt{k}\, 2^{-N\zeta}\le r_2$. Noting that $2^{-n\zeta}<k^{-1/2}2^{\zeta}d(s,t)$,
\begin{equation}\label{result}
\sup_{m\ge 1}|Y_{s,m}-Y_{t,m}| \le C'' 2^{-n\beta} = C'' (2^{-n\zeta})^{\beta/\zeta} \le C''' d(s,t)^{\beta'},
\end{equation}
where $\beta'=\beta/\zeta$ and $C'''=C''k^{-\beta/2\zeta}2^{\beta}$.

We have shown that for $\mathbb{P}$-almost every $\omega\in A_\epsilon$ \eqref{result} holds for all $n \geq N$ for some $N$. Now let
\[
A=\bigcup_{\epsilon>0} \big\{\omega\in A_\epsilon: \eqref{result} \text{ holds}\big\}
\]
As $M$ is almost surely finite  $\mathbb{P}(A)=1$, and for each $\omega \in A$ there exists an $\epsilon>0$ such that $\omega \in A_\epsilon$, hence there is an $N>0$ such that \eqref{result} holds for all $n \geq N$. Finally, to extend \eqref{result} from $\mathcal{T}_*$ to $\mathcal{T}$, we use the continuity of $t\mapsto Y_{t,n}$ for $n\ge 1$ and the fact that $\mathcal{T}_*$ is dense in $\mathcal{T}$. Inequality  \eqref{holderest3} follows by renaming constants appropriately.
\end{proof}

\begin{remark}\label{rema3}
The only point at which condition {\rm (A3)} is used is in the above proof is at the start of the induction where we choose $s_*,t_*\in  \mathcal{T}_{M\zeta}$ such that $d(s_*,t_*)\le d(s,t)$. The argument would remain valid (with changes to the constants) if  {\rm (A3)} is replaced by a weaker but more awkward condition that states the property of $\mathcal{T}$ that is actually used:
\begin{itemize}
{\rm 
\item[(A3$'$)] There exist an increasing sequence of sets of points $\mathcal{T}_1\subset \mathcal{T}_2\subset\cdots$ in $\mathcal{T}$ and constants $C_3,\alpha_3>0$ such that for each $n\ge 1$,  $\# \mathcal{T}_n \le C_3 2^{n\alpha_3}$ and $\{B_d(t,2^{-n}): t\in \mathcal{T}_n\}$ forms a covering of $\mathcal{T}$ such that each point in $\mathcal{T}$ is covered by at most $C_3$ balls. In particular $\mathcal{T}_*:= \bigcup_{n=1}^\infty \mathcal{T}_n$ forms a countable dense subset of $\mathcal{T}$. Furthermore, for all $s,t\in \mathcal{T}$ with $d(s,t)\le C_3 2^{-n}$, and all $m\ge n+1$ there exist $s_m,t_m\in \mathcal{T}_m$ such that $d(s,s_m)\le 2^{-m}$, $d(t,t_m)\le 2^{-m}$ and $d(s_m,t_m)\le C_3 2^{-n}$.}
\end{itemize}
\end{remark}

Recall from Section \ref{secabscty} that $\lambda=\alpha_2\wedge (2\alpha_2')$ and
\[
n(\lambda,k,\gamma)=\Big(\frac{4k^2-\lambda k}{\lambda^2}\Big)\gamma^2+\frac{2k}{\lambda^2}\gamma\sqrt{4k^2\gamma^2+2k(1-\gamma^2)\lambda}+\frac{k}{\lambda},
\]
as well as
\[
m(\alpha_1,\gamma)=\frac{1}{2}\Big(\frac{\alpha_1}{\gamma}-\frac{\gamma}{2}\Big)^2
\]
As before $s_{\alpha_1,\gamma}(p)$ is given by \eqref{defsp}.
\begin{lemma}\label{lem4.4}
When $\alpha_1>\frac{\gamma^2}{2}$ and $k\ge \frac{\lambda}{2}$, the condition of Proposition \ref{prop1}, that there exist $p>1$ and $q>1$ such that 
\begin{equation}\label{pqcond} 
0<\frac{\frac{1}{2}+\frac{\gamma^2}{2}(q-1)}{(\frac{1}{2}\alpha_2)\wedge \alpha_2'-\frac{k}{q}}<\frac{s_{\alpha_1,\gamma}(p)}{k}<\infty,
\end{equation}
is equivalent to
\begin{equation}\label{equivcond}
n(\lambda,k,\gamma)<m(\alpha_1,\gamma).
\end{equation}
\end{lemma}

\begin{proof}
First the minimum
\[
\min\left\{\frac{\gamma^2q^2+(1-\gamma^2)q}{\lambda q-2k}: q>\frac{2k}{\lambda}\right\}
\]
occurs at
\[
q_*=\frac{2k+\sqrt{4k^2+2k(\frac{1}{\gamma^2}-1)\lambda}}{\lambda}
\]
and equals
\[
\Big(\frac{4k-\lambda}{\lambda^2}\Big)\gamma^2+\frac{2}{\lambda^2}\gamma\sqrt{4k^2\gamma^2+2k(1-\gamma^2)\lambda}+\frac{1}{\lambda}.
\]
The equation
\[
(\alpha_1-\gamma^2)\frac{p}{2}=\Big(\alpha_1-\frac{\gamma^2}{2}p\Big)(p-1)
\]
has two solutions
\[
p_0=2 \text{ and } p_1=\frac{\alpha_1}{\gamma^2}.
\]
Therefore if $p_1\le 2$ then $s_{\alpha_1,\gamma}(p)=\Big(\alpha_1-\frac{\gamma^2}{2}p\Big)(p-1)$ for all $p\ge 1$ and if $p_1>2$ then
\[
s_{\alpha_1,\gamma}(p)=\left\{\begin{array}{ll}
\Big(\alpha_1-\frac{\gamma^2}{2}p\Big)(p-1), & \text{ for } 1\le p \le 2;\\
(\alpha_1-\gamma^2)\frac{p}{2}, & \text{ for } 2<p\le p_1;\\
\Big(\alpha_1-\frac{\gamma^2}{2}p\Big)(p-1), & \text{ for } p>p_1.
\end{array}\right.
\]
In either case $s_{\alpha_1,\gamma}(p)=\Big(\alpha_1-\frac{\gamma^2}{2}p\Big)(p-1)$ for $p\ge p_1$, therefore the maximum of $s_{\alpha_1,\gamma}(p)$ always occurs at $p_*=\frac{\alpha_1}{\gamma^2}+\frac{1}{2}>p_1$ and equals 
\[
s_{\alpha_1,\gamma}(p_*)=\frac{1}{2}\Big(\frac{\alpha_1}{\gamma}-\frac{\gamma}{2}\Big)^2.
\]
When $\alpha_1>\frac{\gamma^2}{2}$ and $k\ge \frac{\lambda}{2}$ we have $p_*>1$ and $q_*>1$, which gives the conclusion.
\end{proof}

Our main Theorem \ref{thm1.1} now follows easily.
\medskip

\noindent {\it Proof of Theorem   \ref{thm1.1}}.
If \eqref{mcond} is satisfied then by Lemma \ref{lem4.4} the hypotheses of Proposition \ref{prop1} are satisfied for some $p>1$ and $q>1$. Thus for the value of  $\beta >0$ given by Proposition \ref{prop1}, the sequence of $\beta$-H\"older continuous functions $\{t\mapsto Y_{t,n}\}_{n=1}^\infty$ is almost surely uniformly bounded and equicontinuous. With this value of $p$, Lemma \ref{lemma1a}  and the Borel-Cantelli lemma imply that almost surely for all $t\in \mathcal{T}_*$ the sequence  $\{Y_{t,n}\}_{n=1}^\infty$ is Cauchy and so convergent. Since 
$\mathcal{T}_*$ is dense in $\mathcal{T}$, this pointwise convergence together with the equicontinuity implies that 
 $\{t\mapsto Y_{t,n}\}_{n=1}^\infty$ converges uniformly to some function $t\mapsto Y_t$ which  must be $\beta$-H\"older continuous since the $\{t\mapsto Y_{t,n}\}_{n=1}^\infty$ are uniformly $\beta$-H\"older, as required.
$\Box$
\medskip

Condition \eqref{holderest3}, which leads to the condition \eqref{mcond} for Theorem \ref{thm1.1} to hold and consequently to the restrictions on $\gamma$ in Theorems \ref{projs}, \ref{ft} and \ref{ql}, is unlikely to be best possible. Indeed we might hope for Theorems \ref{projs} and \ref{ft} to be valid for all $0<\gamma < 2-\sqrt{2}$.  The lack of sharpness comes from the estimates in Lemma \ref{lemma2} where we have used the modulus of continuity of circle averages of GFF before estimating the moments; using such almost sure estimates to control the moments typically leads to loss of  sharpness.  Moreover, in Lemma \ref{lemma2}  we work on the preimage of a measure through the functions $f_t$, but the partition of the parameter space ${\mathcal T}$ does not necessarily yield a partition of the space ${\overline D}$, so when estimating the moments of summations we cannot use the von Bahr-Esseen type inequalities. We are working in a very general setting of measures so it is not easy to obtain sharp results as in the case of 1d Lebesgue measure or the occupation measure of planar Brownian motion where the measures have stationarity and scaling invariance properties. Our estimates appear reasonably good given that we use the modulus of continuity before estimating moments, and significant improvements are likely to require new methods.

Following through the proofs would allow an estimate of the H\"{o}lder exponent $\beta$ in Theorem \ref{thm1.1}. This depends on the difference between the two expressions in \eqref{holderest1}. If this difference is $\epsilon$ then one can choose integers $\ell =\lceil \epsilon/3\rceil$ and $\zeta$ such that $\zeta/ \ell$ differs from both expressions by at least 
$\epsilon/3$. This allows for good estimates for $\delta_1$ and $\delta_2$ and thus for $\beta$ as defined after \eqref{est2}. However $\beta$ is redefined after \eqref{est2} by dividing by $\zeta$  leading to a somewhat smaller value of $\beta$ in Theorem \ref{thm1.1}.

\section{Applications of the main theorem - proofs}

This section gives the proofs of the various applications of Theorem \ref{thm1.1} that are stated in Sections  \ref{secquantlen}, \ref{secsss} and \ref{secabscty}.

We first derive Theorem \ref{projs} on the H\"{o}lder continuity of LQG when $D\subset \mathbb{R}^2$ is a rotund convex domain, that is has twice continuously differentiable boundary with radius of curvature bounded away from $0$ and $\infty$.  Such a domain satisfies (A0) since the intersection of two convex sets is convex and so simply connected, with the ball condition holding provided $2^{-n} \leq 2^{-{N_0}} $ is less than the minimum radius of curvature of $\partial D$. We first need a geometrical lemma on the H\"{o}lder continuity of chord lengths of such a domain.

For $ (\theta,u) \in (\mathbb{R}\,{\rm mod }\,\pi)\times \mathbb{R}$ let $l_{(\theta,u)}$  be the straight line in $\mathbb{R}^2$ in direction $\theta$  and perpendicular distance $u$ from the origin. We identify these lines $l_{(\theta,u)}$ with the parameters  $(\theta,u)$ and define a metric $d$ by
\begin{equation}\label{metricd}
d(l_{(\theta,u)},l_{(\theta',u')}) \equiv d\big((\theta,u),(\theta',u')\big) 
= |u-u'| + \min\big\{|\theta-\theta'|, \pi - |\theta-\theta'|\big\}.
\end{equation}
We write $L(l)$ for the length of the chord $l\cap {\overline D}$ provided the line $l$ intersects ${\overline D}$.

\begin{lemma}\label{convex}
Let $D \subset \mathbb{R}^2$ be a rotund convex domain. There is a constant $c_0$ depending only on $D$ such that for all  $l,l'$ that  intersect ${\overline D}$
\begin{equation}\label{lineholder}
\big|L(l)-L(l')\big| \  \leq\  c_0 d(l,l')^{1/2}.
\end{equation}
\end{lemma}

\begin{proof}
It is convenient to work with an alternative geometrical interpretation of the metric $d$. Given a line $l$ and $\epsilon >0$ let $S_\infty(l,\epsilon)$ be the infinite strip $\{ x\in \mathbb{R}^2: |x-y|\leq \epsilon \mbox{ for some } y \in l\}$. For $M>0$ let $R_M(l,\epsilon)$ be the rectangle 
$\{ x\in S_\infty(l,\epsilon): |x\cdot {\theta}| \leq M \}$ where here we regard ${\theta}$ as a unit vector in the direction of $l$ and `$\cdot$' denotes the scalar product. Fix $M$ sufficiently large so that for all lines $l$ and $\epsilon >0$,
$$S_\infty(l,\epsilon)\cap  {\overline D} = R_M(l,\epsilon)\cap  {\overline D}.$$
Write 
$$E_M(l,\epsilon) = \big\{l': l'\cap \partial R_M(l,\epsilon)
= \{x_{-},x_{+}\} \mbox{ where } 
x_\pm\!\cdot{\theta}=\pm M \big\},$$
for the set of lines that  enter and exit  the rectangle $R_M(l,\epsilon)$ across its two `narrow' sides.

\begin{figure}[htp]
\centering
\begin{tikzpicture}
\fill[gray] (-3.25,0.25) -- (-2.75,-0.25) -- (0.25,2.75) -- (-0.25,3.25);
\draw (-3.1,0.4) node [left] {$R_M(l,\epsilon)$};
\draw (-0.7,0.7) node {$D$};
\fill (0cm,0cm) circle [radius=2pt];
\draw[->] (0,0) -- (0.5,0.5) node [above] {$\theta$};
\draw[dashed] (-2.65,1.25) -- (-1.25,2.65);
\draw[<->](-2.2,0.8)--  (-2.65,1.25) node [left] {$d_\parallel (l)$};
\draw[thick] (-3.23,-0.23) -- (0.25,3.25) node [above right] {$l$};
\draw (-3.05,-0.36) -- (0.12,3.46) node [above] {$l'$};
\draw[thick] plot [smooth cycle] coordinates {(-1.95,1.95) (0,2.5) (0.8,0.8) (-0.1,-0.9) (-2.5,0)};
\end{tikzpicture}
\caption{}\label{fig1}
\end{figure}

It is easy to see that there are constants $\epsilon_0, \lambda>0$ depending only on $D$ (taking into account $M$ and the position of $D$ relative to the origin) such that if $d(l,l')\leq \lambda\epsilon\leq \lambda \epsilon_0$ then $l'\in E_M(l,\epsilon)$. Thus \eqref{lineholder} will follow  if there is a constant $c_1$ such that for all $l$ that intersect ${\overline D}$ and all sufficiently small $\epsilon$,
\begin{equation}\label{rect}
\mbox{ if $l'\in E_M(l,\epsilon)$ then } \big|L(l)-L(l')\big| \  \leq\  c_1 \epsilon^{1/2}. 
\end{equation}
Write $0<\rho_{\min} \leq \rho_{\max} <\infty$ for the minimum and maximum radii of curvature of $\partial D$. For a line $l$ that intersects ${\overline D}$ let $d_\parallel (l)$ denote the perpendicular distance between $l$ and the closest  parallel tangent to $\partial D$, see Figure \ref{fig1}. We consider two cases.

\medskip
(a) $ \epsilon \leq  \frac{1}{4}\rho_{\min},\  \frac{1}{2} d_\parallel (l)\leq \epsilon$.
Here both of the `long' sides of the rectangle $R_M(l,\epsilon)$ are within distance 
$d_\parallel (l) + \epsilon \leq 3\epsilon <\rho_{\min}$ of the tangent to $\partial D$  parallel to $l$, so that if $l'\in E_M(l,\epsilon)$ then
$d_\parallel (l') \leq 3\epsilon$. By simple geometry, 
$L(l),L(l') \leq (2 \rho_{\max})^{1/2}(3 \epsilon)^{1/2}$, so \eqref{rect} holds with $c_1 = (2 \rho_{\max})^{1/2}3^{1/2}$.

\medskip
(b) $ \epsilon \leq  \frac{1}{4}\rho_{\min},\  \frac{1}{2} d_\parallel (l)\geq \epsilon$.
In this case, all $l'\in E_M(l,\epsilon)$ are distance at least 
$d_\parallel (l) -\epsilon \geq  \frac{1}{2}\epsilon$ from their parallel tangents to $\partial D$. In particular, the angles between every $l'\in E_M(l,\epsilon)$ and the tangents to $\partial D$ at either end of $l'$ are at least $\phi$ where 
$\cos \phi = \big(\rho_{\max}-\frac{1}{2}\epsilon\big)\big/\rho_{\max}$. 
Both $l,l'\in E_M(l,\epsilon)$ intersect $\partial D$ at points on each of its arcs of intersection with $R_M(l,\epsilon)$, so that $l$ and $l'$ intersect each of these arcs at points within distance
$$\frac{2\epsilon}{\sin\phi} \leq \frac{2\epsilon}{\Big(1 -\big(1-\frac{1}{2}\frac{\epsilon}{\rho_{\max}}\big)^2\Big)^{1/2}}
\leq 2(2\rho_{\max})^{1/2} \epsilon^{1/2}$$
of each other, where we have used $\epsilon/\rho_{\max} \leq \frac{1}{2}$ in the second estimate.
Applying the triangle inequality (twice) to the points  of $l\cap \partial D$ and $l'\cap \partial D$ inequality \eqref{rect} follows with  $c_1 =  4(2\rho_{\max})^{1/2}$.
\end{proof}

\begin{remark}
Note that \eqref{lineholder} remains true taking $d$ to be any reasonable metric on the lines. Moreover, it is easy to obtain a H\"{o}lder exponent of 1 if  we restrict to lines that intersect ${\overline D}\setminus (\partial D)_\delta$ for given $\delta >0$, where $(\partial D)_\delta$ is the $ \delta$-neighbourhood of the boundary of $D$.
\end{remark}

\noindent{\it Proof of Theorem \ref{projs}}.
We first show that the  total mass of GMC-measures of Lebesgue measure restricted to chords $l\cap D$ is H\"{o}lder continuous; we do this by showing that the family of parameterized measures  satisfies the conditions of Theorem \ref{thm1.1}. We have already remarked that $D$ satisfies (A0).

Choose $R$ such that $D \subset B(0,R)$. 
Let $\nu$ be Lebesgue measure on the interval $E=[-R,R]$. Let
$$\mathcal{T}\ =\ \big\{(\theta,u)\in [0,2\pi/3]\times \mathbb{R}: l_{(\theta,u)}\cap \overline{D} \neq\emptyset\big\}.$$
For $(\theta,u)\in \mathcal{T}$ let
\[
I_{(\theta,u)}=\pi^*_{\theta +\pi/2}(l_{(\theta,u)} \cap {\overline D})
\]
where  $\pi^*_{\theta +\pi/2}$ denotes orthogonal projection onto the line $l_\theta$ through $0$ in direction $\theta$ followed by a translation along $l_\theta$ to map the mid-point of $l_{(\theta,u)} \cap {\overline D}$ to $0$; we identify $l_\theta$ with $\mathbb{R}$ in the natural way. Let
\[
f_{(\theta,u)}(v)=u \e^{i(\theta+\pi/2)}+v\e^{i\theta}, \quad v\in I_{(\theta,u)},
\]
where we identify $\mathbb{R}^2$ with $\mathbb{C}$.
Then
\[
\nu_{(\theta,u)}:=\nu\circ f_{(\theta,u)}^{-1}
\]
is just $1$-dimensional Lebesgue measure on the chord $l_{(\theta,u)} \cap {\overline D}$ of $D$. It is easy to see that $(\mathcal{T},d)$ is compact. Also  $\{\nu_{(\theta,u)}: (\theta,u)\in \mathcal{T}\}$ clearly satisfies (A1) for $C_1=1$ and $\alpha_1 =1$. 

For condition (A2), for $(\theta,u),(\theta',u')\in \mathcal{T}$ and $v\in E$,
\begin{eqnarray*}
\big|f_{(\theta,u)}(v)-f_{(\theta',u')}(v)\big| &\le& \big(|v|+|u|\big)\big|1-\e^{i(\theta-\theta')}\big|+|u-u'|\\
&\le&2\sqrt{2}R\big|1-\cos(\theta-\theta')\big|^{1/2}+|u-u'|\\
&\le&2\sqrt{2}R\big( \min\big\{|\theta-\theta'|, \pi - |\theta-\theta'|\big\}+|u-u'|\big)\\
&=& 2\sqrt{2}R\, d(l_{(\theta,u)},l_{(\theta',u')}). 
\end{eqnarray*}
 Also, by Lemma \ref{convex},
$$
\nu\big(I_{(\theta,u)}\Delta I_{(\theta',u')}\big)
= \big| L\big(l_{(\theta,u)}\big)-L\big(l_{(\theta',u')}\big) \big|\ 
\leq\  c_0 d\big(l_{(\theta,u)},l_{(\theta',u')}\big)^{1/2}.
$$
This gives (A2) with $C_2=\max\{ 2\sqrt{2}R, c_0\}$, $\alpha_2=1$ and $\alpha_2'=\frac{1}{2}$.

To check (A3) let $h_+, h_-: [0,2\pi/3] \to \mathbb{R}^+$ be the positive and negative support functions of $D$, i.e.
$$  h_-(\theta) \ =\ \inf \{x\cdot \theta : x\in D\}, \quad  h_+(\theta) \ =\ \sup \{x\cdot \theta : x\in D\},$$
where we identify $\theta$ with a unit vector in the direction $\theta$  and `$\cdot$' is the scalar product.
Then the map:
$$ (\theta, u) \to  \bigg(\frac{3}{2\pi}\theta,\frac{u-h_-(\theta)}{h_+(\theta)-h_-(\theta)}\bigg) $$
is a one-to-one continuously differentiable, and in particular bi-Lipschitz, map from  $\mathcal{T}$ to the convex set $G:=  [0,1] \times [0,1]$, as required.

For $(\theta,u)\in \mathcal{T}$ and $n\ge 1$ let $\widetilde{\nu}_{(\theta,u),n}$ and $Y_{(\theta,u),n}$ be given as in \eqref{defnutn} \and \eqref{defytn}. With $\alpha_1=1$, $\lambda=\alpha_2\wedge(2\alpha_2')=1$ and $k=2$, condition \eqref{mcond} becomes
\[
14\gamma^2+4\gamma\sqrt{12\gamma^2+4}+2\ <\ \frac{1}{2}\Big(\frac{1}{\gamma}-\frac{\gamma}{2}\Big)^2.
\]
This inequality 
implies that $\frac{1}{2}\Big(\frac{1}{\gamma}-\frac{\gamma}{2}\Big)^2>14\gamma^2+2$, which means
\[
\gamma<\frac{1}{111}\sqrt{444\sqrt{34}-1110}\approx \ 0.34645967,
\]
and under this condition the inequality is equivalent to
\[
33\gamma^8+344\gamma^6-488\gamma^4-160\gamma^2+16>0.
\]
The smallest positive zero of this polynomial is
\begin{equation}\label{polysol}
\gamma_*=\frac{1}{33}\sqrt{858-132\sqrt{34}}\ \approx\  0.28477489.
\end{equation}
Therefore if $0< \gamma<0.28$ then \eqref{mcond} is true in this setting, thus the conclusions of Theorem \ref{thm1.1} hold for some $\beta>0$. 
Hence we may assume that, as happens almost surely, $Y_{(\theta,u),n}$ converges uniformly on $\mathcal{T}$ to a $\beta$-H\"{o}lder continuous $Y_{(\theta,u)}$, and  for all $n$ the circle averages $\widetilde{\mu}_{2^{-n}}$ defined with respect to Lebesgue measure $\mu$ on $D$ are absolutely continuous and converge weakly to the $\gamma$-LQG measure $\widetilde{\mu}$. Note that we have shown that $Y_{(\theta,u)}$ is H\"{o}lder continuous for  $(\theta,u)\in  [0,2\pi/3]\times \mathbb{R}$. The same argument applied to  $(\theta,u)\in  [\pi/3,\pi]\times \mathbb{R}$  ensures H\"{o}lder continuity for all  $(\theta,u)\in (\mathbb{R}\ {\rm mod }\,\pi)\times \mathbb{R}$. 

Now fix  $\theta$ and let $(u,v)\in \mathbb{R}^2$ be coordinates in directions $\theta+\frac{\pi}{2}$  and $\theta$. Let $\phi(u,v) \equiv  \phi(u)$ be continuous on $\mathbb{R}^2$ and independent of the second variable. Since $\widetilde{\nu}_{(\theta,u),n}$ are absolutely continuous measures, using \eqref{muep}, \eqref{defnutn} and Fubini's theorem,
\begin{eqnarray*}
\int_{(u,v)\in D}  \phi(u) d\widetilde{\mu}_{2^{-n}}(u,v) 
&=& \int_{(u,v)\in D}  \phi(u) 2^{-n\gamma^2/2} \e^{\gamma \Gamma(\rho_{(u,v),2^{-n}})} dv\, du\\
&=& \int_{(u,v)\in D}  \phi(u) 2^{-n\gamma^2/2} \e^{\gamma \Gamma(\rho_{(u,v),2^{-n}})} d\nu_{(\theta,u)}(v)\, du\\
&=& \int_{u_{-}(\theta)}^{u_{+}(\theta)}  \phi(u) \|\widetilde{\nu}_{(\theta,u),n}\|\ du\\
&=& \int_{u_{-}(\theta)}^{u_{+}(\theta)}  \phi(u) Y_{(\theta,u),n} du,
\end{eqnarray*}
where $u_{-}(\theta)$ and $u_{+}(\theta)$ are the values of $u$ corresponding to the tangents to $D$ in direction $\theta$.
Letting $n \to \infty$ and using the weak convergence of $\widetilde{\mu}_{2^{-n}}$ and the uniform convergence of $Y_{(\theta,u),n}$,
\begin{equation}\label{projint}
 \int_{u_{-}(\theta)}^{u_{+}(\theta)}  \phi(u) d(\pi_\theta\widetilde{\mu})(u) 
 =\int_{(u,v)\in D}  \phi(u) d\widetilde{\mu}(u,v) 
=  \int_{u_{-}(\theta)}^{u_{+}(\theta)}  \phi(u)Y_{(\theta,u)} du.
\end{equation}
Thus $d(\pi_\theta\widetilde{\mu})(u) = Y_{(\theta,u)} du$ on $[u_{-}(\theta),u_{+}(\theta)]$, so as $Y_{(\theta,u)}$ is $\beta$-H\"{o}lder continuous on the interval $[u_{-}(\theta),u_{+}(\theta)]$ we conclude that $\pi_\theta\widetilde{\mu}$ is absolutely continuous with a $\beta$-H\"{o}lder Radon-Nikodym derivative. 
$\Box$
\medskip

Note that for a single {\it fixed} $\theta$ the projected measure $\pi_\theta \widetilde{\mu}$ almost surely has a $\beta$-H\"{o}lder continuous  Radon-Nikodym derivative for some $\beta>0$ if
\[
0<\gamma <\frac{1}{17}\sqrt{238-136\sqrt{2}}\ \approx\ 0.3975137.
\]
This follows in exactly the same way as in the above proof but taking $\mathcal{T}$ to be the 1-parameter family $\big\{u\in  \mathbb{R}: l_{(\theta,u)}\cap {\overline D} \neq\emptyset\big\}$. Then $\alpha_1=1, \lambda=1$ and $k=1$, giving $\gamma_*$ in \eqref{polysol} as $\frac{1}{17}\sqrt{238-136\sqrt{2}}$ in this case.

The decay rate of the Fourier transform of $\gamma$-LQG $\widetilde{\mu}$  follows from the H\"{o}lder continuity of the measures induced by $\widetilde{\mu}$ on slices by chords of $D$.
\medskip

\noindent{\it Proof of Corollary \ref{ft}}.
We use the same notation as in the proof of Theorem \ref{projs} above. Almost surely, $\mathcal{T}\ni (\theta,u)\mapsto Y_{(\theta,u)}$ is $\beta$-H\"older continuous (for $\mathcal{T}$ covers the directions $[0,2\pi/3]$ as well as $[\pi/3,\pi]$) where $\beta$ is given by Theorem \ref{projs}, that is, for some $C_\beta>0$,
\[
|Y_{(\theta,u)}-Y_{(\theta',u')}|\le C_\beta d\big((\theta,u),(\theta',u')\big)^\beta.
\]
Write $[0,\pi]^*=\mathbb{R}\ {\rm mod }\,\pi$. For $\theta\in [0,\pi]^*$ and $j \in\big\{u_{-}(\theta),u_{+}(\theta)\big\}$,
\[
\mathbb{E}\big(\lim_{u\to j} Y_{(\theta,u)}\big)\leq\lim_{u\to j}\mathbb{E}\big(Y_{(\theta,u)}\big) =\lim_{u\to j}\mathbb{E}\big(\|\widetilde{\nu}_{(\theta,u)}\|\big)=0,
\]
since $\lim_{u\to j}\|\nu_{(\theta,u)}\|=0$. As $\lim_{u\to j} Y_{(\theta,u)}\ge 0$, this implies that almost surely the limit $\lim_{u\to j} Y_{(\theta,u)}=0$. Taking a countable dense subset of $[0,\pi]^*$ and applying H\"older continuity, we conclude that almost surely $Y_{(\theta,j)}=0$ for all $\theta\in [0,\pi]$ and $j\in\big\{u_{-}(\theta),u_{+}(\theta)\big\}$. This means that we can extend $Y_{(\theta,u)}$ to all $u \in \mathbb{R}$ by letting $Y_{(\theta,u)}=0$ for $u\notin\big[u_{-}(\theta),u_{+}(\theta)\big]$, with the extended function still $\beta$-H\"older continuous with the same constant $C_\beta$.

Write the transform variable $\xi = \tilde{\xi}\theta$ where here we regard $\theta\in [0,\pi]^*$ as a unit vector and $ \tilde{\xi}\in \mathbb{R}$. From  \eqref{projint}
\begin{equation}\label{fteq}
\widehat{\widetilde{\mu}}(\tilde{\xi} \theta)=
\int_{D} \e^{\i (\tilde{\xi} \theta) \cdot x} \widetilde{\mu}(dx)
= \int_{u_{-}(\theta)}^{u_{+}(\theta)}  \e^{\i  \tilde{\xi} u} d(\pi_\theta\widetilde{\mu})(u) 
=  \int_{u_{-}(\theta)}^{u_{+}(\theta)} \e^{\i  \tilde{\xi} u}  Y_{(\theta,u)} du.
\end{equation}
Let $M> \max\big\{| u_{-}(\theta)|,|u_{+}(\theta)|\big\} +1$.  Then $Y_{(\theta,u)}$ is supported in $[u_{-}(\theta),u_{+}(\theta)]\subset [-M,M]$.   Using an argument attributed to Zygmund, for $\big|\tilde{\xi}\big| >\pi$,
$$
\int_{-M}^{M} \e^{\i  \tilde{\xi} u}  Y_{(\theta,u)} du
= \int_{-M}^{M} \e^{\i  \tilde{\xi}( u+\pi/\tilde{\xi})}  Y_{(\theta,u+\pi/\tilde{\xi})} du
= - \int_{-M}^{M}\e^{\i  \tilde{\xi}u}  Y_{(\theta,u+\pi/\tilde{\xi})} du.
$$
The first and third integrals both equal the transform, so 
 $$\big| \widehat{\widetilde{\mu}}(\tilde{\xi} \theta)\big| 
 =\frac{1}{2} \Big| \int_{-M}^{M} \e^{\i  \tilde{\xi}u} \big[Y_{(\theta,u)}- Y_{(\theta,u+\pi/\tilde{\xi})}\big] du\Big|
 \leq M C_\beta\Big( \frac{\pi}{\tilde{\xi}}\Big)^\beta
$$
by the H\"older condition, giving \eqref{ftest}.
$\Box$
\medskip

Finally we apply  Theorem \ref{thm1.1} to the H\"{o}lder continuity of a family of self-similar measures to get Theorem  \ref{ssm}.
\medskip

\noindent{\it Proof of Theorem \ref{ssm}}.
Take $I_t = E$ for all $t \in\mathcal{T}$ in Theorem \ref{thm1.1}. We claim that $\big\{(g_i^t, I_t): t \in\mathcal{T}\big\}$ satisfies assumptions (A1)-(A3).

A standard estimate using the open set condition shows that 
\begin{equation}\label{fro}
\nu_t(B(x,r))\le C_1 r^{\alpha_1},\quad  x \in \mathbb{R}^2, r>0,
\end{equation}
where $\alpha_1 = \min_{t\in\mathcal{T},1\leq i \leq m} \log p_i / \log r_{i}$ and $C_1>0$ for (A1). Moreover, 
\begin{align*}
|f_s(\underline{i})& -f_t(\underline{i})| \leq 
\lim_{n\to\infty} |g^s_{i_1}\circ\cdots\circ g^s_{i_n}(x_0)-g^t_{i_1}\circ\cdots\circ g^t_{i_n}(x_0)|\\
&\leq \lim_{n\to\infty}\Big\{\big |(g^s_{i_1}-g^t_{i_1})\circ g^s_{i_2}\circ\cdots\circ g^s_{i_n} (x_0)\big|
+ \big|g^t_{i_1}\circ(g^s_{i_2}-g^t_{i_2})\circ g^s_{i_3}\circ\cdots\circ g^s_{i_n} (x_0)\big|\\
&\qquad \qquad + \cdots + \big|g^t_{i_1}\circ\cdots\circ (g^s_{i_n}-g^t_{i_n})(x_0)\big|\Big\}\\
&\leq \sum_{n=0}^\infty r_{+}^n c_0 d(s,t) = C_2 d(s,t),
\end{align*}
using that the $ g^t_{i}$ are uniformly Lipschitz on $\mathcal{T}$ and their contraction ratios are bounded by 
$r_{+}: =\max_{t \in\mathcal{T}, 1\leq i\leq m} \{r_{i} \}<1$.  Trivially $\nu(I_s \Delta I_t)=\nu(\emptyset) =0$, so (A2) is satisfied.
Condition (A3) holds as  $\mathcal{T}$ is a compact subset of the locally Euclidean $4m$-dimensional manifold $\mathcal{U}$. 

Hence the assumptions (A1)-(A3) of Theorem \ref{thm1.1} are satisfied and  Theorem \ref{ssm} follows.
$\Box$

\bigskip

\noindent{\bf Acknowledgement} The authors are most grateful to the Isaac Newton Insitute, Cambridge, for hospitality and support during the Random Geometry Programme where part of this work was done. They are also grateful to Julien Barral and Emilio Corso and to a referee for helpful comments and for pointing out oversights on various drafts of this paper.

\bibliographystyle{abbrv}

\end{document}